\documentclass[a4paper,11pt,reqno,twoside]{amsart}

\usepackage{amsmath}
\usepackage{amsfonts}
\usepackage{amssymb}
\usepackage{amsthm}
\usepackage{color}
\usepackage{ifpdf}
\usepackage{array}
\usepackage{url}
\usepackage{multirow,bigdelim}
\usepackage{ascmac}
\usepackage[dvipdfmx]{graphicx}

\numberwithin{equation}{section}

\newtheorem{theorem}{Theorem}[section]
\newtheorem{lemma}{Lemma}[section]
\newtheorem{proposition}{Proposition}[section]

\newcommand*{\C}{\mathbb{C}}
\newcommand*{\R}{\mathbb{R}}
\newcommand*{\Z}{\mathbb{Z}}
\newcommand*{\N}{\mathbb{N}}

\newcommand{\comment}[1]{}
\title[Aspects of the screw function of $\zeta$]%
      {Aspects of the screw function corresponding to \\ the Riemann zeta-function} 
\author[M. Suzuki]{Masatoshi Suzuki}
\date{Version of \today}
\subjclass[]{
11M26 
42A82 
44A60 
}
\keywords{Riemann zeta-function, screw functions, Stieltjes moment problem}
\AtBeginDocument{%
\begin{abstract}
We introduce a screw function corresponding to the Riemann zeta-function 
and study its properties from various aspects. 
Typical results are several equivalent conditions 
for the Riemann hypothesis in terms of the screw function. 
One of them can be considered an analog of so-called Weil's positivity or Li's criterion.
In addition, we prove a few partial but unconditional results for such equivalents. 
\end{abstract}
\maketitle
}
\begin{document}

%
\section{Introduction} 
%

Let $\zeta(s)$ be the Riemann zeta-function 
and let $\xi(s)$ be the Riemann xi-function. 
The latter is an entire function defined by 
\[
\xi(s) = \frac{1}{2}s(s-1)\pi^{-s/2}\Gamma\left(\frac{s}{2}\right)\zeta(s)
\]
and satisfies two functional equations 
$\xi(s)=\xi(1-s)$ and $\xi(s)=\overline{\xi(\bar{s})}$, 
where $\Gamma(s)$ is the gamma-function and the bar denotes the complex conjugate.

Typical results of the present paper are several equivalent conditions 
for the Riemann hypothesis (RH, for short) which claims that 
all zeros of $\xi(s)$ lie on the critical line $\Re(s)=1/2$.
%
The core of the interrelation among 
all such equivalents is the function $\Psi(t)$ on $[0,\infty)$ defined by  
\begin{equation} \label{Eq_101}
\aligned 
\Psi(t)
& := 4(e^{t/2}+e^{-t/2}-2) - \sum_{n \leq e^t} \frac{\Lambda(n)}{\sqrt{n}}(t-\log n)  \\
&\quad 
+ \frac{t}{2}\left[ \frac{\Gamma'}{\Gamma}\left(\frac{1}{4}\right) - \log \pi \right] 
+ \frac{1}{4}\left( C - e^{-t/2}\Phi(e^{-2t},2,1/4) \right),  
\endaligned 
\end{equation}
where $\Lambda(n)$ is the von Mangoldt function defined by 
$\Lambda(n)=\log p$ if $n=p^k$ with $k \in \Z_{>0}$ 
and $\Lambda(n)=0$ otherwise,  
$C = \pi^2 + 8G$ with the Catalan constant 
$G = \sum_{n=0}^{\infty} (-1)^n(2n+1)^{-2}$, 
and $\Phi(z,s,a) = \sum_{n=0}^{\infty} (n+a)^{-s}z^n$ 
is the Hurwitz--Lerch zeta function. 
Formula \eqref{Eq_101} shows that $\Psi(t)$ is real-valued and continuous on $[0,\infty)$ 
and that $\Psi(0)=0$ by $\Phi(1,2,1/4) =\zeta(2,1/4)= C$. 
First, we state a few fundamental properties of $\Psi(t)$ that are unconditionally proven 
in Section \ref{Section_2} below. 

\begin{theorem} \label{Thm_1_1} 
The following holds for $\Psi(t)$ defined in \eqref{Eq_101}. 
\begin{enumerate}
\item The one-sided Fourier transform formula 
\begin{equation} \label{Eq_102}
\int_{0}^{\infty} \Psi(t) \, e^{izt} \, dt 
= -\frac{1}{z^2}\frac{\xi'}{\xi}\left(\frac{1}{2}-iz \right)
\end{equation}
holds if $\Im(z)>1/2$ , where $i=\sqrt{-1}$.  
\item The series representation
\begin{equation} \label{Eq_103}
\aligned 
\Psi(t)
& = \sum_{\gamma} \frac{1-\cos(\gamma t)}{\gamma^2} 
= \sum_{\gamma} \frac{1-e^{i\gamma t}}{\gamma^2} 
\endaligned 
\end{equation}
holds for every $t \geq 0$, 
where the sums in the middle and the right-hand side range over 
all zeros $\gamma$ of $\xi(1/2-iz)$ counting with multiplicity. 
\item The estimate $\Psi(t)  \ll \exp(t/2 -c\sqrt{t})$ holds for some constant $c>0$. 
\end{enumerate}
\end{theorem}

In Theorem \ref{Thm_1_1} and in what follows, 
we use the Vinogradov symbol $f(x) \ll g(x)$ and the Landau symbol $f(x)=O(g(x))$   
as symbols to mean that there exists a positive constant $M$ such that 
$|f(x)| \leq M g(x)$ holds for prescribed range of $x$. 
\medskip

By the first equality in \eqref{Eq_103}, the function $\Psi(t)$ is naturally extended 
to an even function on the real line. 
Therefore, 
we henceforth identify $\Psi(t)$ with that extended even function, that is, 
we understand by replacing $t$ with $-t$ in the right-hand side of \eqref{Eq_101} 
when $t$ is negative. 
Also, by Theorem \ref{Thm_1_1}, any of \eqref{Eq_101}, \eqref{Eq_103}, or 
the Fourier inversion of \eqref{Eq_102} can be chosen as the definition of $\Psi(t)$, 
but in this paper, we chose definition \eqref{Eq_101} including prime numbers.
\medskip

Formulas \eqref{Eq_102} and \eqref{Eq_103} suggest that 
the function $\Psi(t)$ is related to the class of screw functions 
introduced by M. G. Kre\u{\i}n. 
For $0<a \leq \infty$, he introduced the class $\mathcal{G}_a$ 
consisting of all continuous functions $g(t)$ on $(-2a,2a)$ such that 
$g(-t) = \overline{g(t)}$ (hermitian) 
and the kernel 
\begin{equation} \label{Eq_104}
G_g(t,u):=g(t-u)-g(t)-g(-u)+g(0).  
\end{equation}
is non-negative definite on $(-a,a)$, that is,  
\begin{equation} \label{Eq_105}
\sum_{i,j=1}^{n} G_g(t_i,t_j) \,  \xi_i \overline{\xi_j} \,\geq\, 0
\end{equation}
for all $n \in \N$, $\xi_i \in \C$, and $|t_i| < a$, $(i = 1, 2, . . . , n)$. 
(In literature, a kernel satisfying \eqref{Eq_105} is often referred 
to as a positive definite kernel or semi-positive definite kernel,
but in this paper, we use the term above.) 
The members of $\mathcal{G}_a$ are called {\it screw functions} 
because of their relationship with screw arcs in Hilbert spaces 
(\cite[\S12]{KrLa14}). 

Let $\mathcal{N}$ be the Nevanlinna class
that consists of analytic functions in the upper half-plane $\C_+=\{z\,|\, \Im(z)>0\}$ 
mapping $\C_+$ into $\C_+ \cup \R$. 
(Note that $\mathcal{N}$ is also called the class of Pick functions, 
or R functions, 
or Herglotz functions 
depending on the literature.)   
Kre\u{\i}n--Langer \cite[Satz 5.9]{KrLa77} (\cite[Prop. 5.1]{KrLa85}) 
showed that the equality 
\begin{equation*} 
\int_{0}^{\infty} g(t)\,e^{izt} \, dt = -\frac{i}{z^2}Q(z), \quad \Im(z) > h
\end{equation*}
for some $h \geq 0$ 
establishes a bijective correspondence 
between all functions $g \in \mathcal{G}_\infty$ with $g(0)=0$ 
and all functions $Q \in \mathcal{N}$ with the property 
\begin{equation} \label{Eq_106}
\lim_{y \to +\infty} \frac{Q(iy)}{y} =0.
\end{equation}
On the other hand, 
J. C. Lagarias~\cite[(1.5)]{La99} 
proved that the RH is true if and only if
\begin{equation*}
\Im\left[ i\, \frac{\xi'}{\xi}\left(\frac{1}{2}-iz\right) \right] >0 \quad \text{when} \quad \Im(z)>0. 
\end{equation*}
The latter means that the function 
\begin{equation} \label{Eq_107}
Q_\xi(z):=i\, \frac{\xi'}{\xi}\left(\frac{1}{2}-iz\right)
\end{equation}
belongs to $\mathcal{N}$. 
It is easy to confirm that 
$Q_\xi(z)$ satisfies \eqref{Eq_106} unconditionally 
by Dirichlet series expansion of $(\zeta'/\zeta)(s)$ 
and an asymptotic expansion of $(\Gamma'/\Gamma)(s)$ as $|s| \to \infty$ 
(see \eqref{Eq_410}). 
Therefore, if we assume that the RH is true, 
$Q_\xi(z) $ belongs to $\mathcal{N}$ and satisfies \eqref{Eq_106}, 
thus there exists a corresponding screw function $g \in \mathcal{G}_\infty$. 
Such $g$ must be equal to the function defined 
by
\begin{equation} \label{Eq_108}
g(t):= -\Psi(t)
\end{equation}
from the Fourier integral formula \eqref{Eq_102} 
and the uniqueness of the Fourier transform. 
Conversely, assuming that this $g(t)$ is a screw function, 
the RH holds from \eqref{Eq_102} 
and the above result of Kre\u{\i}n--Langer. 

Hence we obtain the following equivalent condition 
for the RH which is the starting point 
of other equivalent conditions for the RH in the present paper.

\begin{theorem} \label{Thm_1_2}
The RH is true if and only if $g(t)$ defined by \eqref{Eq_108} 
is a screw function on $\R=(-\infty,\infty)$. 
\end{theorem}

Henceforth throughout this paper, 
$g(t)$ is the function defined by \eqref{Eq_101} and \eqref{Eq_108}, 
and $G_g(t,u)$ is the kernel defined by \eqref{Eq_104} for this $g(t)$, 
unless stated otherwise. 
We call $g(t)$ the {\it screw function of the Riemann zeta-function} 
after Theorem \ref{Thm_1_2}, 
although it is an abuse of words in a strict sense. 
The series representation of $g(t)$ obtained from \eqref{Eq_103} 
is nothing but an integral representation of a screw function 
(\cite[Theorem 5.1 and (7.11)]{KrLa14}) and implies 
\begin{equation} \label{Eq_109}
G_g(t,u) 
 = \sum_{\gamma} \frac{(e^{i\gamma t}-1)(e^{-i\gamma u}-1)}{\gamma^2}
\end{equation}
(cf. the second line of the proof of \cite[Theorem 5.1]{KrLa14}). 
The kernel $G_g(t,u)$ is non-negative on $\R^2$ under the RH, 
but it can be shown unconditionally that 
$G_g(t,u)$ is non-negative if $|t| < a$ and $|u| < a$ for small $a>0$ 
(see Section \ref{Section_4}). 
\medskip

As a main application of Theorem \ref{Thm_1_2}, 
we obtain a variant of Weil's famous 
criterion for the RH by the positivity of the Weil distribution 
(see Section \ref{subsec_Weil}) in terms of screw functions 
as follows. 
For each $0<a \leq \infty$, 
we define the hermitian form $\langle \cdot,\cdot \rangle_{G_g,a}$ 
for functions supported in $[-a,a]$ by 
\begin{equation} \label{Eq_110}
\langle \phi_1,\phi_2 \rangle_{G_g,a}
:= \int_{-a}^{a}\int_{-a}^{a}G_g(t,u)\phi_1(u)\overline{\phi_2(t)} \, dudt
\end{equation}
when the right-hand side is absolutely convergent. 
We also define the space 
\begin{equation} \label{Eq_111}
\mathfrak{C}_0(a) =\left\{ \phi \in C_c^\infty(\R)\,\left|~{\rm supp}\,\phi \subset[-a,a],
~\int_{-a}^{a} \phi(t)\, dt=0 \right.\right\}, 
\end{equation}
where $C_c^{\infty}(\R)$ 
is the space of all smooth functions on the real line having compact support. 
Note that $\mathfrak{C}_0(a)$ is not the class of continuous functions on $(-a,a)$.

\begin{theorem}  \label{Thm_1_3}
The RH is true if and only if the hermitian form
$\langle \cdot,\cdot \rangle_{G_g,a}$ 
is non-negative definite on $\mathfrak{C}_0(a)$, 
that is, $\langle \phi,\phi \rangle_{G_g,a} \geq  0$ 
for all $\phi \in \mathfrak{C}_0(a)$ for every $0<a<\infty$. 
Moreover, assuming that the RH is true, 
$\langle \cdot,\cdot \rangle_{G_g,a}$ 
is positive definite on $L^2(-a,a)$, 
that is, 
$\langle \phi,\phi \rangle_{G_g,a}>0$ 
for all non-zero $\phi \in L^2(-a,a)$ for every $0<a<\infty$. 
\end{theorem}

See Section \ref{Section_3} for details and more on the Weil distribution 
and Weil's positivity criterion. 
In addition to Theorem \ref{Thm_1_3}, 
we obtain the following analog of 
the equivalent of the RH by H. Yoshida \cite{Yo92} 
(see also Section \ref{subsec_Yoshida})
described by the nondegeneracy of hermitian forms. 

\begin{theorem}  \label{Thm_1_4}
The RH is true if and only if the hermitian form
$\langle \cdot,\cdot \rangle_{G_g,a}$ 
is non-degenerate on $L^2(-a,a)$ for every $0<a<\infty$. 
The latter condition is equivalent that the integral operator 
$\mathsf{G}_g[a]: L^2(-a,a) \to L^2(-a,a)$ defined by 
\begin{equation} \label{eq_110}
\mathsf{G}_g[a]:~\phi(t) ~\mapsto~ \mathbf{1}_{[-a,a]}(t) \int_{-a}^{a} G_g(t,u) \phi(u) \, du,
\end{equation}
does not have zero as an eigenvalue for every $0<a<\infty$, 
where $\mathbf{1}_A(t)$ is the characteristic function of a subset $A \subset \R$. 
\end{theorem}

The advantage of Theorems \ref{Thm_1_3} and \ref{Thm_1_4} 
is that hermitian forms $\langle \cdot,\cdot \rangle_{G_g,a}$ are represented 
by an integral operator with a continuous kernel acting on usual $L^2$-spaces. 
This makes the strategy of Yoshida \cite{Yo92} and E. Bombieri \cite{Bo01} 
for the RH via the Weil distribution analytically more straightforward and simple.
\medskip

It is well known that the continuity of the kernel 
is not sufficient to conclude that the corresponding 
integral operator is of trace class.  
If we assume that the RH is true, $\mathsf{G}_g[a]$ is a positive definite 
Hilbert--Schmidt operator for every $0<a < \infty$. 
Therefore, it is a trace class operator by Mercer's theorem. 
Furthermore, the traceability of $\mathsf{G}_g[a]$ 
does not depend on the definiteness of $\mathsf{G}_g[a]$ as follows. 
\bigskip

\begin{theorem} \label{Thm_1_5}
For each $0<a<\infty$, $\mathsf{G}_g[a]$ is a trace class operator 
unconditionally. 
\end{theorem}

Since $\mathsf{G}_g[a]$ is a trace class operator, 
\[
{\rm Tr}\,\mathsf{G}_g[a] 
= \sum_{n=1}^{\infty} \lambda_{a,n} 
= \int_{-a}^{a} G_g(t,t) \, dt
\]
holds by \cite[Chapter IV, Theorem 8.1]{GGK01}, 
where $\lambda_{a,1}, \lambda_{a,2},\dots$ 
are non-zero eigenvalues of $\mathsf{G}_g[a]$ counting with multiplicity. 
\bigskip

We state further applications of Theorem \ref{Thm_1_2} 
which are somewhat secondary to Theorems \ref{Thm_1_3} and \ref{Thm_1_4}, 
but they are interesting in their own right, 
or the relations between different subjects suggested by those results are interesting. 
\medskip

The function $\Psi(t)$ is bounded on $[0, \infty)$ by \eqref{Eq_103} 
under the RH, and vice versa. 

\begin{theorem} \label{Thm_1_6}
The RH is true if and only if $\Psi(t) = O(1)$ on $[0, \infty)$. 
\end{theorem}

If $\Psi(t)$ belongs to $L^1(0,\infty)$ or $L^2(0,\infty)$, 
formula \eqref{Eq_102} holds for $\Im(z)>0$ and defines 
an analytic function in $\C_+$ 
whose extension to the real line $\Im(z)=0$ 
is bounded or a function of $L^2(\R)$, respectively. 
Therefore, in both cases, it contradicts $z=0$ 
being a simple pole of the right-hand side. 
Hence, $\Psi(t)$ belongs to neither $L^1(0,\infty)$ nor $L^2(0,\infty)$, 
regardless of whether the RH is true or false.
\medskip

If the RH is true, $G_g(t,t)=2(g(0)-g(t))=2\Psi(t)$ is non-negative 
by \eqref{Eq_105} for $n=1$ and $\xi=1$, and vice versa. 

\begin{theorem} \label{Thm_1_7}
The RH is true if and only if 
$\Psi(t)$ is pointwise non-negative, that is, 
$\Psi(t) \geq 0$ for every $t \in \R$. 
Further, if RH is true, $\Psi(t)>0$ when $t \not= 0$. 
\end{theorem}

We then discretize the pointwise positivity condition in Theorem \ref{Thm_1_7}  
using the $n$th moment
\begin{equation} \label{Eq_113}
\mu_n := \int_{0}^{\infty} 4^{-1}e^{-t/2}\,\Psi(t) \cdot t^n \, dt 
\end{equation}
for $n \in \Z_{\geq 0}$, where the integral 
is absolutely convergent by Theorem \ref{Thm_1_1} (3).

\begin{theorem} \label{Thm_1_8}
Let 
\[
\Delta_n := 
\begin{pmatrix}
\mu_0 & \mu_1 & \cdots & \mu_{n} \\
\mu_1 & \mu_2 & \cdots & \mu_{n+1} \\
\vdots & \vdots & \ddots & \vdots \\
\mu_{n} & \mu_{n+1} & \cdots & \mu_{2n} \\
\end{pmatrix}
\quad \text{and} \quad 
\Delta_n^{(1)} := 
\begin{pmatrix}
\mu_1 & \mu_2 & \cdots & \mu_{n+1} \\
\mu_2 & \mu_3 & \cdots & \mu_{n+2} \\
\vdots & \vdots & \ddots & \vdots \\
\mu_{n+1} & \mu_{n+2} & \cdots & \mu_{2n+1} \\
\end{pmatrix}
\]
be Hankel matrices consisting of moments $\mu_k$ 
defined by \eqref{Eq_113}. 
Then the RH is true if and only if $\det \Delta_n \geq 0$ and $\det \Delta_n^{(1)} \geq 0$ 
for all $n \in \Z_{>0}$. 
\end{theorem}

In \cite{Li97}, X.-J. Li proved the so-called 
Li's criterion that claims that the RH is true 
if and only if all Li coefficients defined by 
\begin{equation} \label{Eq_114}
\lambda_{n+1}=\frac{1}{n!}\frac{d^n}{dw^n}\left[
\frac{1}{(1-w)^2} \, \frac{\xi'}{\xi}\left(\frac{1}{1-w}\right)
\right]_{w=0}, 
\end{equation}
are positive. 
Since the sequence  $\{n^{-1}\lambda_n\}_{n=1}^{\infty}$ was considered 
by J. B. Keiper \cite{Ke92} before Li, 
the Li coefficients are also referred to as the Keiper--Li coefficients 
in some literature. 
Bombieri--Lagarias \cite{BoLa99} found that 
the positivity of all Li coefficients is a discretization 
of the positivity of the Weil distribution. 
The relation of Theorems \ref{Thm_1_3}  and \ref{Thm_1_8} 
can be considered as an analog of the relation of the Weil distribution and Li coefficients 
(see also Section \ref{Section_3_4}). 
We are then interested in a direct relation between 
moments $\mu_n$ and Li coefficients $\lambda_n$. 
Changing of variables as $z=(i/2)(1-2X)$ in \eqref{Eq_102} 
and then expanding $\exp(izt)$ 
to the power series of $X$, we obtain
\begin{equation} \label{Eq_115}
\mu_n = 
\frac{d^n}{dX^n}\left[
\frac{1}{(1-2X)^2}\frac{\xi'}{\xi}(1-X)
\right]_{X=0}. 
\end{equation}
The similarity between \eqref{Eq_114} and \eqref{Eq_115} is obvious, 
and it is shown in Section \ref{Section_8} that 
there is an explicit relation between them.
\smallskip

Recall \eqref{Eq_107} and define $q_\xi(z):=Q_\xi(\sqrt{z})/\sqrt{z}$. 
Then, $q_\xi(z)$ 
belongs to the subclass of $\mathcal{N}$ corresponding to Kre\u{\i}n's strings 
under the RH (see Section \ref{Section_9}).  
The corresponding string is the one named Zeta string by S. Kotani \cite{Ko21}. 

On the other hand, the series representation \eqref{Eq_103} 
suggests that $g(t)=-\Psi(t)$ is a mean-periodic function 
by appropriate choice of function space. 
In fact, it is unconditionally shown that $g(t)$ is mean-periodic 
using the function spaces used in Fesenko--Ricotta--Suzuki \cite{FRS12} (see Section \ref{Section_10}).
\smallskip

The functions in class $\mathcal{G}_a$ have various fruitful connections 
with many different mathematical objects, for example, 
due to a relation with positive-definite functions described in \cite{KrLa14}. 
Therefore, more interesting discoveries and connections are expected 
for the screw function of the Riemann zeta-function 
from the origin mentioned above. 
Furthermore, although this paper focused only on the Riemann zeta-function for simplicity, 
it would be natural to extend the results of this paper 
to other general $L$-functions like automorphic $L$-functions 
or $L$-functions in the Selberg class.  
However, we leave such studies for 
other papers \cite{NaSu22, Su22a, Su22b, Su23a, Su23b} and future research.
\smallskip

This paper is organized as follows. 
In Section \ref{Section_2}, we prove Theorem \ref{Thm_1_1} and a lemma needed in later sections. 
In Section \ref{Section_3}, we prove Theorem \ref{Thm_1_3} 
after reviewing the Weil distribution and preparing for the notation. 
In Section \ref{Section_4}, we establish the pointwise non-negativity of $\Psi(t)$, 
and the non-negativity of $G_g(t,u)$ for small $t$, $u$, 
as a preliminary step towards proving Theorem \ref{Thm_1_4}. 
Moreover, we state and prove a lower bound 
for $\langle \phi, \phi \rangle_{G_g,a}$ 
under restrictions to $\phi$. 
In Section \ref{Section_5}, we prove Theorem \ref{Thm_1_4} 
using results in Section \ref{Section_4} 
and make a comparison with Yoshida's results.
In Section \ref{Section_6}, 
we prove Theorem \ref{Thm_1_5} 
and study eigenvalues of the kernel $G_g(t,u)$. 
In Section \ref{Section_7}, we prove Theorems \ref{Thm_1_6} and \ref{Thm_1_7}, and \ref{Thm_1_8}. 
In Section \ref{Section_8}, 
we describe explicit relations between the moments $\mu_n$ and the Li coefficients.  
In Section \ref{Section_9}, 
we describe a Kre\u{\i}n string corresponding to $Q_\xi(z)$ in \eqref{Eq_107} 
under the RH. 
In Section \ref{Section_10}, 
we discuss the mean-periodicity of the screw function $g(t)$. 
\comment{
In Section \ref{section_12}, 
we note that the zeros of $\xi(1/2-iz)$ 
are interpreted as eigenvalues of a self-adjoint operator 
acting on a Hilbert space obtained by the completion 
with respect to the Hermitian form $\langle \cdot,\cdot\rangle_{G_g,\infty}$ 
under the RH. 
}
In Section \ref{Section_11}, we introduce variants $\Psi_\omega(t)$ of $\Psi(t)$ 
and state an analog of Theorem \ref{Thm_1_7}. 
Further, we study relations among different $\Psi_\omega(t)$'s. 
\medskip

\noindent
{\bf Acknowledgments}~
The author appreciates the valuable suggestions and comments of the referees
that resulted in improvements in the presentation and organization of the paper.  
This work was supported by JSPS KAKENHI Grant Number JP17K05163 
and JP23K03050. 
This work was also supported by the Research Institute for Mathematical Sciences, 
an International Joint Usage/Research Center located in Kyoto University.

%
\section{Proof of Theorem \ref{Thm_1_1} and a Lemma} \label{Section_2}
%

\subsection{Proof of Theorems \ref{Thm_1_1} (1)} 

We use the change of variables $s=1/2-iz$ for convenience of writing. 
We first note that the equality 
\begin{equation*} 
\aligned 
\frac{\xi'}{\xi}(s) 
& = \frac{1}{s-1} +\frac{1}{s} - \frac{1}{2}\log \pi 
+ \frac{1}{2}\frac{\Gamma'}{\Gamma}\left(\frac{s}{2}\right) 
- \sum_{p} \log p \sum_{n=1}^{\infty} p^{-ns}
\endaligned 
\end{equation*}
holds for $\Re(s)>1$, which is equivalent to $\Im(z)>1/2$ in terms of $z$. 
We have 
\begin{equation} \label{Eq_201} 
\int_{0}^{\infty} 4(e^{t/2}+e^{-t/2}-2)\, e^{izt} \, dt 
= -\frac{1}{z^2}\,\left(\frac{1}{s-1}+\frac{1}{s}\right) \quad \text{if}~\Im(z)>1/2, 
\end{equation}
\begin{equation} \label{Eq_202} 
\int_{0}^{\infty} t\,e^{izt} \, dt 
= -\frac{1}{z^2} \quad \text{if}~\Im(z)>0, \quad \text{and}
\end{equation}
\begin{equation*} 
\int_{0}^{\infty} \frac{(t-\log n)}{\sqrt{n}} \, \mathbf{1}_{[\log n,\infty)}(t)\, e^{izt} \, dt 
= -\frac{1}{z^2} \, n^{-s} \quad \text{if}~\Im(z)>0
\end{equation*}
by direct and simple calculation. The third equality leads to 
\begin{equation} \label{Eq_203} 
\aligned 
\int_{0}^{\infty} & 
\left(-\sum_{n \leq e^t} \frac{\Lambda(n)}{\sqrt{n}}(t-\log n)\right) e^{izt} \, dt \\
& = -
\int_{0}^{\infty} \sum_{n=2}^{\infty}  
\frac{\Lambda(n)}{\sqrt{n}}(t-\log n)\mathbf{1}_{[\log n,\infty)}(t)\, e^{izt} \, dt \\
& = 
-\frac{1}{z^2}\left(-\sum_{n=2}^{\infty} \Lambda(n)\, n^{-s} \right)\quad  \text{if}~
\Im(z)>1/2. 
\endaligned 
\end{equation} 
Hence, the proof is completed if it is proved that 
\begin{equation} \label{Eq_204}
\int_{0}^{\infty} 
\frac{1}{4}
\left( C - e^{-t/2}\Phi(e^{-2t},2,1/4) \right) e^{izt} \, dt 
=-
 \frac{1}{2z^2} \left(
\frac{\Gamma'}{\Gamma}\left(\frac{s}{2}\right)
-\frac{\Gamma'}{\Gamma}\left(\frac{1}{4}\right)
\right)
\end{equation}
holds for $\Im(z)>0$. To prove this, we recall the well-known series expansion 
\begin{equation} \label{Eq_205} 
\frac{\Gamma'}{\Gamma}(w) = -\gamma_0 - \sum_{n=0}^{\infty}
\left( \frac{1}{w+n} - \frac{1}{n+1} \right), 
\end{equation}
where $\gamma_0$ is the Euler--Mascheroni constant. 
On the other hand, we have 
\begin{equation*} 
\int_{0}^{\infty} \frac{1-e^{-2(n+\frac{1}{4}) t}}{2(n+\frac{1}{4})^2}\, e^{izt} \, dt 
= \frac{1}{z^2} \left( \frac{1}{s/2+n} - \frac{1}{n+1/4} \right) 
\end{equation*}
for non-negative integers $n$ and $\Im(z)>0$ 
by direct and simple calculation.  
In addition, noting that $C=\pi^2+8G$ is a special value of 
the Hurwitz zeta function, 
precisely $C=\zeta(2,1/4)=\Phi(1,2,1/4)=\sum_{n=0}^{\infty}(n+1/4)^{-2}$, 
we get
\begin{equation*} 
\aligned 
\int_{0}^{\infty} & 
\frac{1}{4}\Bigl[
C - 
e^{-t/2} 
\Phi(e^{-2t},2,1/4)
\Bigr] e^{izt}\, dt \\
& = \frac{1}{2}
\int_{0}^{\infty} 
\left( 
\sum_{n=0}^{\infty}
\frac{1-e^{-2(n+\frac{1}{4}) t}}{2(n+1/4)^2}
\right) e^{izt}\, dt  
 = \frac{1}{2z^2}  \sum_{n=0}^{\infty} \left( \frac{1}{s/2+n} - \frac{1}{n+1/4} \right) \\
& = \frac{1}{2z^2}  \sum_{n=0}^{\infty} \left( \frac{1}{s/2+n} - \frac{1}{n+1} \right)
+ 
\frac{1}{2z^2}  \sum_{n=0}^{\infty} \left(  \frac{1}{n+1} - \frac{1}{n+1/4}   \right) \\
& = 
-
 \frac{1}{2z^2} \left(\frac{\Gamma'}{\Gamma}\left(\frac{s}{2}\right) + \gamma_0 \right) 
+ \frac{1}{2z^2} \left(\frac{\Gamma'}{\Gamma}\left(\frac{1}{4}\right)+ \gamma_0 \right). 
\endaligned 
\end{equation*}
We used \eqref{Eq_205} in the last equation.
Hence, we obtain equality \eqref{Eq_204}. \hfill $\Box$

\subsection{Proof of Theorems \ref{Thm_1_1} (2)} 

Since $\xi(s)$ is an order one entire function, Hadamard's factorization theorem gives
\[
\xi\left(\frac{1}{2}-iz\right) = e^{a+bz} \prod_\gamma \left[\left(1-\frac{z}{\gamma} \right)
e^{\frac{z}{\gamma}} \right]. 
\]
Taking the logarithmic derivative of both sides, 
\begin{equation} \label{Eq_206}
\frac{\xi'}{\xi}\left(\frac{1}{2}-iz\right)
= ib + i\sum_\gamma \left( \frac{1}{z-\gamma} + \frac{1}{\gamma} \right), 
\end{equation}
where the sum on the right-hand side converges absolutely and uniformly on every
compact subset of $\C$ outside the zeros $\gamma$. 
Substituting $z=0$ into \eqref{Eq_206}, we have $ib=(\xi'/\xi)(1/2)$. 
On the other hand, taking the logarithmic derivative of $\xi(s)=\xi(1-s)$, 
$(\xi'/\xi)(s) = - (\xi'/\xi)(1-s)$. 
Substituting $s=1/2$ in this equality, we get $(\xi'/\xi)(1/2)=0$, 
thus $ib=(\xi'/\xi)(1/2)=0$. For each term on the right-hand side of \eqref{Eq_206}, 
we have 
\begin{equation*} 
-\frac{i}{z^2}  \left( \frac{1}{z-\gamma} + \frac{1}{\gamma} \right)
=
\int_{0}^{\infty} \frac{1-e^{-it\gamma}}{\gamma^2}\, e^{izt} \, dt, 
\quad \Im(z)>\Im(\gamma)
\end{equation*}
by direct calculation of the right-hand side, 
where $|\Im(\gamma)| \leq 1/2$ (\cite[Theorem 2.12]{Tit86}). 

Therefore, if $\tilde{\Psi}(t)$ 
is the function defined by the right-hand side of \eqref{Eq_103}, 
then \eqref{Eq_102} holds when $\Im(z)>1/2$ 
by interchanging summation and integration, 
which is justified 
by the absolute convergence of the sum 
on the right-hand side of \eqref{Eq_103}, 
and noting the symmetry $\gamma \to -\gamma$ 
coming from the functional equation $\xi(s)=\xi(1-s)$. 
The two functions $\Psi(t)$ and $\tilde{\Psi}(t)$ 
have the same Fourier transform, so they are identical. 
The second equality in \eqref{Eq_103} also follows from 
the symmetry $\gamma \to -\gamma$. 
\hfill $\Box$

\subsection{Proof of Theorems \ref{Thm_1_1} (3)} 

The assertion aimed at follows from the following proposition and \eqref{Eq_101}. 
The result is weaker than the one 
that follows from the best known 
zero-free region of $\zeta(s)$,  
but it is sufficient to guarantee the convergence of the integral in \eqref{Eq_113}. 

\begin{proposition} \label{Prop_2_1} There exists $c>0$ such that 
\begin{equation} \label{Eq_207}
\sum_{n \leq e^t} \frac{\Lambda(n)}{\sqrt{n}}(t-\log n) 
= 4e^{t/2} + O\left(e^{t/2} e^{-c\sqrt{t}} \,\right)
\end{equation}
holds for $t>0$. 
\end{proposition}
\begin{proof} 
We write the left-hand side of \eqref{Eq_207} as $\varphi(t)$. 
For $X>0$ and $c>0$, we have 
\begin{equation*} 
\aligned 
- \frac{1}{2\pi} & \int_{-X+ic}^{X+ic} \frac{e^{-izt}}{z^2} \, dz 
= 
\begin{cases}
\displaystyle{ t + O\left( e^{ct}\,\min\left\{ \frac{1}{|t|X^2},\,\frac{1}{\sqrt{X^2+c^2}} \right\} \right)} 
 & (t>0), \\[10pt]
\displaystyle{ 0 + O\left( \frac{2X}{X^2+c^2} \right)} 
 & (t=0), \\[10pt]
\displaystyle{ 0 + O\left( e^{ct}\,\min\left\{ \frac{1}{|t|X^2},\,\frac{1}{\sqrt{X^2+c^2}} \right\} \right)} 
 & (t<0) \\
\end{cases}
\endaligned 
\end{equation*}
by a standard way of analytic number theory as in \cite[\S3.12]{Tit86}. 
Applying this to each term of the Dirichlet series expansion 
$-(\zeta'/\zeta)(s)=\sum_{n=1}^{\infty}\Lambda(n)n^{-s}$, 
we have 
\begin{equation} \label{Eq_208}
\aligned 
- \frac{1}{2\pi} & \int_{-X+ic}^{X+ic} 
\frac{1}{z^2}
\left[ -\frac{\zeta'}{\zeta}\left(\frac{1}{2}-iz\right) \right]  e^{-izt} \, dz
 = \varphi(t) + O\left( \frac{\Lambda(e^t)}{e^{t/2}}\frac{2X}{X^2+c^2} \right) \\
& \qquad  \quad 
+ O\left(
\sum_{n \not= e^t} \frac{\Lambda(n)}{\sqrt{n}} e^{c(t-\log n)}
\min\left\{ \frac{1}{|t-\log n|X^2},\,\frac{1}{\sqrt{X^2+c^2}} \right\}
\right)
\endaligned 
\end{equation}
for $c>1/2$. In the remaining part of the proof, we take $c$ as $1/2+1/t$.

In the sum on the right-hand side of \eqref{Eq_208}, 
the contribution from $n$ in the range $n \leq e^{t}/2$ or $n \geq 3e^{t}/2$ is
\begin{equation} \label{Eq_209}
\ll \frac{1}{X^2}\sum_{n} \frac{\Lambda(n)}{\sqrt{n}} e^{c(t-\log n)}
= \frac{e^{ct}}{X^2} \left[ -\frac{\zeta'}{\zeta}\left(\frac{1}{2}+c\right) \right] 
\ll \frac{t\cdot e^{t/2}}{X^2}, 
\end{equation}
since $|t-\log n| \gg 1$, $e^{ct}=e\cdot e^{t/2}$, and $(\zeta'/\zeta)(1/2+c)=O(t)$ 
(by $\zeta(s)=(s-1)^{-1}+O(1)$ near $s=1$).  
For $n$ in the range $e^t/2 < n < e^t -2$, if we put $[e^t]-n=r$, 
\begin{equation*} 
t - \log n = \log \frac{e^t}{n} = \log \frac{e^t}{[e^t]-r} 
\geq  - \log \left( 1 - \frac{r}{[e^t]} \right) \geq \frac{r}{[e^t]}. 
\end{equation*}
Therefore, the contribution from $n$ in this range is
\begin{equation} \label{Eq_210}
\aligned 
\ll & \frac{e^t}{X^2}
\sum_{2 \leq r < e^{t}/2} \frac{\Lambda([e^t]-r)}{\sqrt{[e^t]-r}} \cdot \frac{1}{r} 
\ll \frac{t \cdot e^t}{X^2}
\sum_{2 \leq r < e^{t}/2} \frac{1}{\sqrt{[e^t]-r}} \cdot \frac{1}{r} \\ 
& \ll \frac{t \cdot e^{t/2}}{X^2}
\sum_{2 \leq r < e^{t}/2} \frac{1}{r} 
\ll \frac{t^2 \cdot e^{t/2}}{X^2}. 
\endaligned 
\end{equation}
The same applies to contributions from $n$ in the range $e^t +2 < n < 3e^t/2$. 
Finally, the contribution from $n$ in the range $e^t-2 < n < e^t +2$ is 
\begin{equation} \label{eq_00516_9}
\ll 
\frac{\Lambda(e^t)}{\sqrt{e^t}} e^{c \, \log \frac{1}{1+O(e^{-t})} }
\frac{1}{X} \ll \frac{t \cdot e^{-t/2}}{X}, 
\end{equation}
because 
\begin{equation*} 
\min\left\{ \frac{1}{|t-\log n|X^2},\,\frac{1}{\sqrt{X^2+c^2}} \right\} 
\leq \frac{1}{\sqrt{X^2+c^2}} \leq \frac{1}{X}. 
\end{equation*}

From the estimates \eqref{Eq_209}, \eqref{Eq_210}, and \eqref{eq_00516_9} 
for the sum in \eqref{Eq_208}, 
we obtain 
\begin{equation*} 
\aligned 
\varphi(t) 
& = 
- \frac{1}{2\pi} \int_{-X+ic}^{X+ic} 
\frac{1}{z^2}
\left[ -\frac{\zeta'}{\zeta}\left(\frac{1}{2}-iz\right) \right]  e^{-izt} \, dz + O\left( \frac{t^2 \cdot e^{t/2}}{X^2} \right)
+ O\left(
\frac{t \cdot e^{-t/2}}{X}
\right)
\endaligned 
\end{equation*}
for $c=1/2+1/t$. Referring the zero-free region 
$\sigma>1-c(\log|t|+3)^{-1}$ ($c>0$) of $\zeta(\sigma+it)$ \cite[Theorem 3.8]{Tit86}, 
we put $\psi(x)=c'(\log|t|+3)^{-1}$ ($0<c'<c$) 
and let $C_X$ be the curve from 
$-X+i(1/2-\psi(-X))$ to $X +i(1/2-\psi(X))$ along $y=1/2-\psi(x)$, 
where $z=x+iy$, 
Then, moving the path of integration downward, 
\begin{equation*} 
\aligned 
\varphi(t) 
& = 4e^{t/2} + O\left( \frac{t^2 \cdot e^{t/2}}{X^2} \right)
+ O\left(
\frac{t \cdot e^{-t/2}}{X}
\right)
 \\
&  - \frac{1}{2\pi} 
\left(
\int_{-X+ic}^{-X+i(\frac{1}{2}-\psi(-X))}
+ 
\int_{C_X} 
+
\int_{X+i(\frac{1}{2}-\psi(-X))}^{X+ic}
\right)
\frac{1}{z^2}
\left[ -\frac{\zeta'}{\zeta}\left(\frac{1}{2}-iz\right) \right]  e^{-izt} \, dz. 
\endaligned 
\end{equation*}
The first and third integrals on the right-hand side are estimated as 
$\ll X^{-2} e^{ct} \log X \ll X^{-2} e^{t/2} \log X$ 
by the estimate $(\zeta'/\zeta)(\sigma+it) \ll \log (|t|+3)$ (\cite[(3.11.7)]{Tit86}). 
The integral on $C_X$ is estimated as 
\begin{equation*} 
\aligned 
\ll & \int_{-X}^{X} \frac{\log(|x|+3)}{(|x|+3)^2} \,e^{t(1/2-\psi(x))} \,dx
\ll 
e^{t/2} \exp\left( - \frac{c't}{\log(X+3)} \right)
\int_{-X}^{X} \frac{\log(|x|+3)}{(|x|+3)^2} \,dx \\
\ll & e^{t/2} \exp\left( - \frac{c't}{\log(X+3)} \right) \log(X+3).
\endaligned 
\end{equation*}
Therefore, we obtain 
\begin{equation*} 
\aligned 
\varphi(t) 
& = 4e^{t/2} 
+ O\left(\frac{e^{t/2} \log X}{X^2} \right)
+ O\left(e^{t/2} \exp\left( - \frac{c't}{\log(X+3)} \right) \log(X+3) \right)\\
& \qquad + O\left( \frac{t^2 \cdot e^{t/2}}{X^2} \right)
+ O\left(
\frac{t \cdot e^{-t/2}}{X}
\right). 
\endaligned 
\end{equation*}
Finally, choosing $X$ as $\log X = t^{1/2}$, we obtain \eqref{Eq_207} 
for a properly taken $c>0$. 
\end{proof}

\subsection{Lemma} 

The following lemma is often used in subsequent sections.

\begin{lemma} \label{Lem_2_1}
Let $F(z)$ be an entire function of the exponential type 
and let $A$ be a complex number. 
Suppose that $F(\gamma)=A$ for all zeros $\gamma$ of $\xi(1/2-iz)$. 
Then $F(z)=A$ as a function.  
\end{lemma}
\begin{proof} It suffices to deal only with case of $A=0$. 
The number of zeros of $F(z)$ in the disc $|z| \leq r$ counted with multiplicity 
is $O(r)$ by the assumption. 
On the other hand, 
the number of distinct zeros of $\xi(1/2-z)$ in the disc $|z| \leq r$ 
is not $O(r)$ by \cite[Section 9.12]{Tit86}, which was first proved by Littlewood. 
This is a contradiction if $F\not=0$. 
\end{proof} 

The known result that the number of simple zeros of $\zeta(s)$ 
on the critical line up to height $T$ is bounded below by $T\log T$ 
is more convenient for the proof of Lemma \ref{Lem_2_1} 
(see Conrey \cite{Con89} and his comments on the result of Levinson (1974) in the introduction).

%
\section{Proof of Theorem \ref{Thm_1_3} } \label{Section_3}
%

We denote the Fourier transform of $\phi(t)$ as 
\begin{equation*}
\widehat{\phi}(z) 
=
\int_{-\infty}^{\infty} \phi(t) \, e^{izt} \, dt.
\end{equation*}
in this and the latter sections even when $z$ is not a real number.

\subsection{Proof of necessity} 
Assuming that the RH is true, $G_g(t,u)$ is non-negative definite on the real line 
by Theorem \ref{Thm_1_2}. 
Therefore, $ \langle \phi, \phi \rangle_{G_g,a}=\langle \mathsf{G}_g[a]\phi, \phi \rangle_{L^2} \geq 0$ 
for every $\phi \in L^2(-a,a)$, 
since all eigenvalues of $\mathsf{G}_g[a]$ are non-negative  (\cite[\S8]{St76}). 
Moreover, the positive definiteness is directly proved as follows.

For integrable functions $\phi_1(t)$ and $\phi_2(t)$ 
with ${\rm supp}\,\phi_i \subset [-a,a]$ ($i=1,2$), 
\begin{equation} \label{Eq_301}
\aligned 
\langle \phi_1, \phi_2 \rangle_{G_g,a}
& 
= \sum_{\gamma} 
\frac{\widehat{\phi_1}(-\gamma)-\widehat{\phi_1}(0)}{\gamma} \cdot 
\frac{\widehat{\overline{\phi_2}}(\gamma)-\widehat{\overline{\phi_2}}(0)}{\gamma}
\endaligned 
\end{equation}
by \eqref{Eq_109}.  
If we take $\phi=\phi_1=\phi_2$ and 
suppose that the RH is true. 
Then all $\gamma$ are real, and thus 
$\widehat{\overline{\phi}}(\gamma)=\overline{\widehat{\phi}(-\gamma)}$. 
Therefore the right-hand side of \eqref{Eq_301} is equal to 
\begin{equation} \label{Eq_302}
\sum_{\gamma } ~
\left| \frac{\widehat{\phi}(\gamma)-\widehat{\phi}(0)}{\gamma} \right|^2
 \geq 0.
\end{equation}
The sum is positive for every non-zero $\phi(t)$ in $L^2(-a,a)$ 
by Lemma \ref{Lem_2_1}, since $\widehat{\phi}(z)$ 
is a non-constant entire function of the exponential type.  
Hence the hermitian form 
$\langle \cdot, \cdot \rangle_{G_g,a}$ is positive definite 
for every $0<a<\infty$ under the RH. \hfill $\Box$
\medskip

To prove the sufficiency of the non-negative definiteness of 
the hermitian forms $\langle \cdot, \cdot \rangle_{G_g,a}$, 
we use Weil's positivity criterion. 

\subsection{Weil's positivity} \label{subsec_Weil} 

The linear functional $W:C_c^\infty(\R)\to\C$ 
defined by 
\begin{equation} \label{Eq_303}
\psi~\mapsto~W(\phi):=\sum_{\gamma} \widehat{\psi}(\gamma) 
\end{equation}
is called the Weil distribution. A. Weil~\cite{We52} (see also \cite{Yo92}) 
showed that the RH is true if and only if 
the distribution $W$ is non-negative definite, that is,  
\begin{equation} \label{Eq_304}
W(\psi \ast \widetilde{\psi}) \geq  0 \quad \text{for every}~\psi \in C_c^\infty(\R),
\end{equation}
where 
\begin{equation} \label{Eq_305}
(\psi_1\ast\psi_2)(x):=\int_{-\infty}^{\infty} \psi_1(y)\psi_2(x-y) \, dy, \qquad 
\widetilde{\psi}(x) := \overline{\psi(-x)}. 
\end{equation}
Note that Weil does not mention that it is sufficient to taking 
compactly supported functions as test functions 
at least in \cite{We52, We72} and his collected papers, 
but in Yoshida \cite{Yo92}, the criterion has been formulated in the form above, 
although it is not sure whether it is the first literature.

\subsection{Reduction of sufficiency to Weil's positivity} 

For $0<a<\infty$, we define
\begin{equation} \label{Eq_306}
C(a) =\left\{ \phi \in C_c^\infty(\R)\,\left|~{\rm supp}\,\psi \subset[-a,a] \right.\right\}. 
\end{equation}
Let we set $D=d/dt$ and  
\begin{equation*}
(I_{b}^{(a)}\phi)(t):=\int_{-a}^{t} \phi(u)\,du + b \quad (b \in \C). 
\end{equation*}
Then the maps 
\begin{equation} \label{Eq_307}
C(a)~\overset{D}{\longrightarrow}~\mathfrak{C}_0(a), \qquad 
\mathfrak{C}_0(a)~\overset{I_{0}^{(a)}}{\longrightarrow}~C(a) 
\end{equation}
are bijective and are inverse to each other, 
where  $\mathfrak{C}_0(a)$ is the space defined in \eqref{Eq_111}. 
Therefore, from the following proposition, 
we find that 
the nonnegativity of $\langle \cdot, \cdot \rangle_{G_g,a}$ 
on $\mathfrak{C}_0(a)$ for every $0<a<\infty$ 
implies that the RH is true.

\begin{proposition} \label{Prop_3_1}
Let $0<a<\infty$ and let $C(a)$ be the space defined in \eqref{Eq_306}. Then, 
\begin{equation} \label{Eq_308} 
\langle D\psi_1, D\psi_2 \rangle_{G_g,a} =W(\psi_1\ast \widetilde{\psi_2})
\end{equation}
for every $\psi_1, \psi_2 \in C(a)$.  
\end{proposition}
\begin{proof} If $\phi=D\psi$ and $\widehat{\phi}(0)=0$, we have 
$z^{-1}(\widehat{\phi}(z)-\widehat{\phi}(0)) = -i\widehat{\psi}(z)$ 
by integration by parts. 
Therefore, 
$\langle D\psi_1, D\psi_2 \rangle_{G_g,a} = \sum_\gamma \widehat{\psi_1}(-\gamma)
\widehat{\overline{\psi_2}}(\gamma)$ by \eqref{Eq_301}. 
The right-hand side is equal to $W(\psi_1\ast \widetilde{\psi_2})$ 
by \eqref{Eq_303} and \eqref{Eq_305}. 
\end{proof}

\subsection{Relation between $W$ and $\Psi$} \label{Section_3_4}

The pointwise positivity of Theorem \ref{Thm_1_7} 
turns out to be a special case of the Weil positivity. 
For $t>0$, we consider the triangular function
\begin{equation} \label{Eq_309}
\Delta_t(x)
=
\begin{cases}
\displaystyle{\frac{1}{2}(t-|x|)}, & |x| \leq t; \\[5pt]
0, &  |x|>t.
\end{cases}
\end{equation}
Then $W(\Delta_t)=\Psi(t)$, because 
\begin{equation} \label{Eq_310} 
\int_{-\infty}^{\infty} \Delta_t(x) \, e^{izx}  dx = \frac{1-\cos(zt)}{z^2} 
\end{equation}
for any $z \in \C$. The triangular function $\Delta_t$ 
is expressed as $\Delta_t = R_t \ast \widetilde{R_t}$ 
by the rectangular function
\begin{equation*}
R_t(x)
=
\begin{cases}
\displaystyle{\frac{1}{\sqrt{2}}}, & |x| < t/2; \\[8pt]
\displaystyle{\frac{1}{2\sqrt{2}}}, & |x| = t/2; \\[8pt]
0, &  |x|>t/2.
\end{cases}
\end{equation*}
Therefore, 
\begin{equation} \label{Eq_311}
\Psi(t) = W(\Delta_t) = W(R_t \ast \widetilde{R_t}). 
\end{equation}
This one is similar to the formula of Li coefficients 
by the Weil distribution $W$ in \cite{BoLa99}.

\subsection{Relation between $W$ and $g$} \label{Section_2_5}

For a test function $\phi \in C_c^\infty(\R)$, 
\begin{equation*}
\int_{-\infty}^{\infty} (-g(t))\phi''(t) \, dt 
= \sum_{\gamma} \frac{1}{\gamma^2} ( \widehat{\phi''}(\gamma) - \widehat{\phi''}(0) ) 
= \sum_{\gamma} \widehat{\phi}(\gamma) = W(\phi),  
\end{equation*}
where the prime means differentiation. Therefore we find that 
\begin{equation*}
-g''(t)=\Psi''(t) = \sum_\gamma e^{i\gamma t} = W(t)
\end{equation*}
as a distribution, where we understand 
the distribution as $W(\phi) = \int W(t) \phi(t) \, dt$. 
From this relation, the Weil distribution $W(t)$ can be regarded as 
an ``accelerant'' of the screw function  $g(t)$ of $\xi(s)$ 
in the sense of \cite[Introduction]{KrLa85}. 

%
\section{Preparations for the proof of Theorem \ref{Thm_1_4}} \label{Section_4}
%

\subsection{Positivity of $\Psi(t)$} 

It is observed that $\Psi(t)$ is positive in a small range $[0,t_0)$ of $t$ 
by numerical calculation by a computer (Figure \ref{fig_01}). 
This observation is unconditionally proven 
and used to prove Theorem \ref{Thm_1_4}.

\begin{figure}
 \begin{minipage}{12cm}
 \centering
\includegraphics[width=12cm]{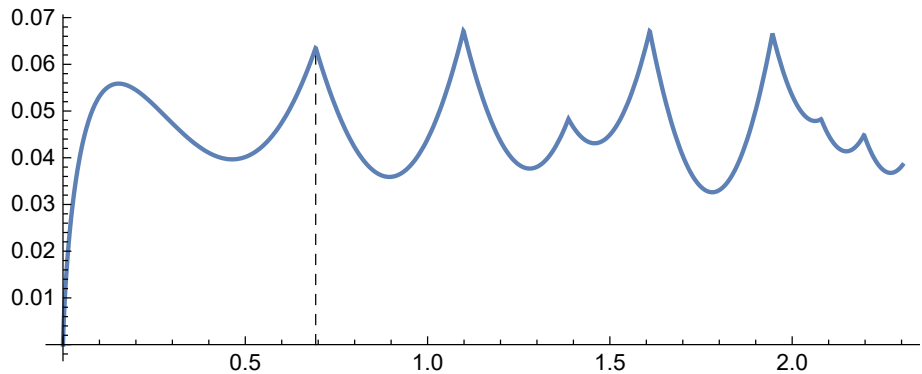}
\caption{Plot of $\Psi(t)$ in the range $0 \leq t \leq \log 10$. 
The vertical dashed line is located at $t=\log 2$.}
\label{fig_01}
  \end{minipage} 
\end{figure}

\begin{theorem} \label{Thm_4_1}
There exists $t_0 > \log 2$ such that $\Psi(t)>0$ for $0< t < t_0$.  
\end{theorem}
\begin{proof} 
This follows from Theorem \ref{Thm_4_2} below for the weaker result $t_0>0$, 
which is sufficient for Theorem \ref{Thm_1_4}, 
but here we prove it in another direct way with the help of numerical calculation by computer.
Differentiating \eqref{Eq_101} and 
noting  $(\Gamma'/\Gamma)(1/4)=-\gamma_0-(\pi/2)-3\log 2$, 
we have 
$\Psi'(t) =  2(e^{t/2}-e^{-t/2})  +c - {\rm arctan}(e^{t/2}) + {\rm arctanh}(e^{-t/2})$
with $c= (\pi/4) - (\gamma_0+3\log 2)/2$, 
since 
there is no contribution from the primes for $\Psi(t)$ on $[0,\log 2)$. 
Then we find that all zeros of $\Psi'(t)$ on  $[0,\log 2)$ are 
$t_1=0.152631\dots$ and $t_2=0.464002\dots$ by numerical calculation.  
Since ${\rm arctanh}(e^{-t/2}) \to + \infty$ as $t \to 0+$, 
the value of $\Psi'(t)$ is positive on $[0,t_1)$ and $(t_2,\log 2)$ and is negative on $(t_1,t_2)$. 
Further, $\Psi(0)=0$ by definition and $\Psi(t_2)=0.0396618 \dots>0$. 
Therefore, $\Psi(t)$ is positive on $[0,\log 2]$ by the continuity. 
\end{proof}

\subsection{Yoshida's results} \label{subsec_Yoshida}

Yohida~\cite{Yo92} studied a hermitian form on $C_c^\infty(\R)$ defined by 
\[
\langle \phi,\psi \rangle_W := W(\phi \ast \widetilde{\psi}).
\] 
The space $C(a)$ in \eqref{Eq_306} was first introduced by him 
to localize the positivity \eqref{Eq_304}. 
For $0<a<\infty$ and $N \in \Z_{\geq 0}$, 
he also introduced the spaces 
\begin{equation*} 
K(a) =\left\{ 
\psi(t)= 
\begin{cases}
h(t), &~(|t| \leq a)~\text{for some $h \in C^\infty(\R)$ with the period $2a$}, \\
0, &~(|t|>a) 
\end{cases}\right\}
\end{equation*}
and 
\begin{equation} \label{Eq_401}
K_N(a) = \left\{ \psi \in K(a) \left|~\int_{-a}^{a} \psi(x)\exp\left(\frac{\pi i n x}{a}\right) \, dx=0~
\text{for all $n \in \Z$, $|n| \leq N$}
\right.\right\}. 
\end{equation}

Yoshida proved without any assumption that 
the hermitian form $\langle \cdot,\cdot \rangle_W$ is positive definite on $K(a)$ 
if $a>0$ is sufficiently small (\cite[Lemma 2]{Yo92}). 
Connes--Consani \cite{CC21a} provides
 an operator theoretic conceptual reason for this result. 
Yoshida also proved that, for given $a_0>0$ and $\mu>0$, there exists $N \geq 0$ such that 
$\langle \phi,\phi \rangle_W \geq \mu \Vert \phi \Vert_{L^2}$ 
for every $\phi \in K_N(a)$ and $0<a \leq a_0$ (\cite[Lemma 3]{Yo92})

\subsection{Positivity of $G_g(t,u)$} 

If the RH is true, $g(t)$ of \eqref{Eq_101} and \eqref{Eq_108} 
belongs to the class $\mathcal{G}_\infty$ by Theorem \ref{Thm_1_2}. 
Therefore, the kernel $G_g(t,u)$ is non-negative definite on $\R$, 
but we can directly confirm that it is non-negative definite on $\R$ as 
\begin{equation*}
\aligned 
\sum_{i,j} G_g(t_i,t_j) \xi_i \overline{\xi_j}
& = 
\sum_{\gamma} \frac{1}{\gamma^2}
\left| \sum_{i} (1-e^{i\gamma t_i})\xi_i\right|^2 \geq 0 
\endaligned 
\end{equation*}
by using \eqref{Eq_109}, 
since all zeros $\gamma$ of $\xi(1/2-iz)$ are real. 
Hence the restriction $g|_{[-2a,2a]}$ belongs to the class $\mathcal{G}_a$ 
for every $a>0$.

As an analog of Yoshida's result above, 
we unconditionally prove that 
the restriction $g|_{[-2a,2a]}$ belongs to $\mathcal{G}_a$  
when $a>0$ is sufficiently small. 
If $g|_{[-2a,2a]}$ belongs to $\mathcal{G}_a$, 
the kernel $G_g(t,u)$ is non-negative definite on $(-a,a)$. 
It implies the nonnegativity of $\Psi(t)$ on $(-a,a)$, because 
\begin{equation*}
0 \leq G_g(t,t) = g(0) - g(t) - g(-t) = 0 - 2g(t) = 2\Psi(t)
\end{equation*}
for $t \in (-a,a)$. In particular, Theorem \ref{Thm_4_1} 
in a weak form $t_0=a>0$ follows from this. 

\begin{theorem} \label{Thm_4_2}
There exists $a_0>0$ such that 
$G_g(t,u)$ is positive definite on $L^2(-a,a)$  if  $0< a < a_0$. 
In particular, the restriction $g|_{[-2a,2a]}$ belongs to the class $\mathcal{G}_a$ 
for every  $0< a < a_0$. 
\end{theorem}
\begin{proof} 
We first introduce a few auxiliary functions. 
For $\phi \in L^2(-a,a)$, we define the transformation 
\begin{equation} \label{Eq_402} 
\Phi_1(\phi,z) := \frac{\widehat{\phi}(z)-\widehat{\phi}(0)}{z} = \int_{-a}^{a} \phi(t) \frac{e^{izt}-1}{z} \, dt. 
\end{equation}
For reals $t$, $u$, and a non-zero real $y$, we define 
\[
\aligned 
K(t,u) : & = \frac{1}{2\pi} \int_{-\infty}^{\infty} 
\frac{e^{izt}-1}{z}\frac{e^{-izu}-1}{z}
 \, dz, \\
K(t,u;y) : & = 
\frac{1}{2\pi}  \int_{-\infty+iy}^{\infty+iy} 
\frac{e^{izt}-1}{z}\frac{e^{-izu}-1}{z}
 \, dz. 
\endaligned
\]
Then, we find that 
\begin{equation} \label{Eq_403}  
\aligned 
K(t,u) 
& = \frac{1}{2} (|t|+|u|-|t-u|) 
= 
\begin{cases}
~u, & \text{if}~t>0,~u>0,~t-u>0, \\
~t, & \text{if}~t>0,~u>0,~t-u<0, \\
~0, & \text{if}~t<0,~u>0,~t-u<0, \\
~-u, & \text{if}~t<0,~u<0,~t-u<0, \\
~-t, & \text{if}~t<0,~u<0,~t-u>0, \\
~0, & \text{if}~t>0,~u<0,~t-u>0, \\
\end{cases}
\endaligned
\end{equation}
and 
\begin{equation} \label{Eq_404}  
\aligned 
K(t,u;y) 
& = 
\frac{1}{2y} \Bigl( -e^{-y(t+|t|)} -e^{-y(u+|u|)} + e^{-y(t+u+|t-u|)} + 1 \Bigr) \\
 \\
& = 
\begin{cases}
~\displaystyle{\frac{1-e^{-2uy}}{2y}}, & \text{if}~t>0,~u>0,~t-u>0, \\[8pt]
~\displaystyle{\frac{1-e^{-2ty}}{2y}}, & \text{if}~t>0,~u>0,~t-u<0, \\[8pt]
~0, & \text{if}~t<0,~u>0,~t-u<0, \\
~\displaystyle{\frac{e^{-2uy}-1}{2y}}, & \text{if}~t<0,~u<0,~t-u<0, \\[8pt]
~\displaystyle{\frac{e^{-2ty}-1}{2y}}, & \text{if}~t<0,~u<0,~t-u>0, \\[8pt]
~0, & \text{if}~t>0,~u<0,~t-u>0, \\
\end{cases}
\endaligned
\end{equation}
by an elementary calculus. 

We then prove the equality 
\begin{equation*} 
\langle \phi, \phi \rangle_{G_g,a} = I_0 + I_1 + I_\infty
\end{equation*}
with 
\begin{equation} \label{Eq_405} 
I_0
 =
 \frac{1}{2\pi} \int_{-\infty+ic}^{\infty+ic} 
\left(
\frac{1}{s-1} + \frac{1}{s}
 \right)
\Bigl(
 \Phi_1(\phi;z)\overline{\Phi_1(\phi;\bar{z})}
+
 \Phi_1(\phi;-z)\overline{\Phi_1(\phi;-\bar{z})}
\Bigr) 
\, dz, 
\end{equation}
\begin{equation} \label{Eq_406} 
I_1
 =
 \frac{1}{2\pi} \int_{-\infty+ic}^{\infty+ic} 
\frac{\zeta'}{\zeta}(s) 
\Bigl(
 \Phi_1(\phi;z)\overline{\Phi_1(\phi;\bar{z})}
+
 \Phi_1(\phi;-z)\overline{\Phi_1(\phi;-\bar{z})}
\Bigr) 
\, dz, 
\end{equation}
\begin{equation} \label{Eq_407}  
I_\infty
 =
\frac{1}{2\pi} \int_{-\infty}^{\infty} 
\Re\left[
\frac{1}{2}\frac{\Gamma'}{\Gamma}\left(\frac{s}{2}\right)- \frac{1}{2}\log \pi 
 \right] 
\Bigl(
|\Phi_1(\phi;z)|^2
+
|\Phi_1(\phi;-z)|^2
\Bigr)
 \, dz,
\end{equation}
where $s=1/2-iz$ as elsewhere. We put
\begin{equation} \label{Eq_408}  
\aligned 
g_0(t) & := - 4(e^{t/2}+e^{-t/2}-2), \quad 
g_1(t) := \sum_{n \leq e^{|t|}} \frac{\Lambda(n)}{\sqrt{n}}(|t|-\log n),  \\
g_\infty(t) & :=  
- \frac{|t|}{2}\left[ \frac{\Gamma'}{\Gamma}\left(\frac{1}{4}\right) - \log \pi \right] -
\frac{1}{4}\left( C - e^{-|t|/2}\Phi(e^{-2|t|},2,1/4) \right) 
\endaligned 
\end{equation}
so that $I_\ast=\langle \phi,\phi \rangle_{G_{g_\ast},a}$ for $\ast \in \{0,1,\infty\}$ 
with the same notation as \eqref{Eq_110}
and $\langle \phi,\phi \rangle_{G_{g},a}=I_0+I_1+I_\infty$. 
We will prove \eqref{Eq_405}, \eqref{Eq_406}, and \eqref{Eq_407} 
for $I_0$, $I_1$, and $I_\infty$, respectively.  

For $t>0$ and $c>1/2$, 
taking the inversion of \eqref{Eq_201} 
and 
noting $g_0(0)=0$,    
\begin{equation*}
g_0(t)
= \frac{1}{2\pi} \int_{-\infty+ic}^{\infty+ic} 
\left(\frac{1}{s-1} + \frac{1}{s} \right)\, \frac{e^{-izt}-1}{z^2} \, dz. 
\end{equation*}
Because we find that the integral on the right-hand side is zero for $t<0$ 
by moving $c \to +\infty$, 
we can write
\begin{equation*}
\aligned 
g_0(t)
&= \frac{1}{2\pi} \int_{-\infty+ic}^{\infty+ic} 
\left(\frac{1}{s-1} + \frac{1}{s} \right)\,
\left( \frac{e^{izt}-1}{z^2} + \frac{e^{-izt}-1}{z^2} \right) \, dz
\endaligned 
\end{equation*}
for any $t \in \R$. Therefore, 
\begin{equation*}
\aligned 
G_{g_0}(t,u)
& = g_0(t-u) - g_0(t) - g_0(-u) +g_0(0) \\
& = \frac{1}{2\pi} \int_{-\infty+ic}^{\infty+ic} 
\left(
\frac{1}{s-1} + \frac{1}{s}
 \right)
\left(
\frac{e^{izu}-1}{z}\frac{e^{-izt}-1}{z} 
+ \frac{e^{-izu}-1}{z}\frac{e^{izt}-1}{z}
\right)
 \, dz. 
\endaligned 
\end{equation*}
From this formula and definitions \eqref{Eq_110} and \eqref{Eq_402}, 
we obtain \eqref{Eq_405}. 
In the same way, \eqref{Eq_406} and \eqref{Eq_407} are obtained 
by using \eqref{Eq_202}, \eqref{Eq_203}, and \eqref{Eq_204}, 
but for \eqref{Eq_407}, we move the horizontal line of the integration 
to the real line in the final step by noting the absence of poles $(\Gamma'/\Gamma)(s/2)$ 
in $\Re(s)>0$. 

Now we prove the positivity of $\langle \phi, \phi \rangle_{G_g,a}$ for small $a>0$. 
We first note the equality 
\begin{equation} \label{Eq_409}
\aligned 
\frac{1}{2\pi} \int_{-\infty}^{\infty} 
& |\Phi_1(\phi;z)|^2
 \, dz \\
& = 
\frac{1}{2\pi} 
\int_{-\infty}^{\infty} 
\left( \int_{-a}^{a} \phi(u) \, \frac{e^{iuz}-1}{z} \, du \right) 
\left( \int_{-a}^{a} \overline{\phi(t)} \, \frac{e^{-itz}-1}{z} \, dt \right) \, dz \\
&= 
\int_{-a}^{a} \int_{-a}^{a} \phi(u)\overline{\phi(t)}\,K(t,u)
\, dudt.
\endaligned 
\end{equation}

We fix $c>1$ in \eqref{Eq_405} and \eqref{Eq_406} and set 
\[
C_1 = \max_{\Re(s)=1+c}\left\{
\left| \frac{1}{s-1} + \frac{1}{s} \right|, ~
\left| \frac{\zeta'}{\zeta}(s) \right|
\right\}. 
\]
Then, for $F(s)=((s-1)^{-1}+s^{-1})$ or $(\zeta'/\zeta)(s)$, 
\[
\aligned 
\, & \left| \frac{1}{2\pi} \int_{-\infty+ic}^{\infty+ic} 
F(s)\, \Phi_1(\phi;\pm z)\overline{\Phi_1(\phi;\pm\bar{z})}
 \, dz 
\right| \\
& \quad \leq C_1 \cdot 
\left[ \frac{1}{2\pi} \int_{-\infty+ic}^{\infty+ic} 
|\Phi_1(\phi; \pm z)|^2  
 \, dz
\right]^{1/2}
\left[ \frac{1}{2\pi} \int_{-\infty+ic}^{\infty+ic} 
|\Phi_1(\phi;\pm \bar{z})|^2  
 \, dz
\right]^{1/2}
\endaligned 
\]
by the Schwartz inequality. 
For the quantities on the right-hand side, we have 
\[
\aligned 
\frac{1}{2\pi} \int_{-\infty + ic}^{\infty + ic} 
|\Phi_1(\phi; \pm z)|^2
 \, dz
&= 
\int_{-a}^{a} \int_{-a}^{a} \phi(u)\overline{\phi(t)}\,K(t,u;\pm c) \, dudt
\endaligned 
\]
as in \eqref{Eq_409}. 
If we fix a real number $a_1>0$, there exists $C_2=C_2(a_1)>0$ such that 
\[
K(t,u; \pm c) \leq C_2 K(t,u) 
\]
for any $|t|\leq a$ and $|u| \leq a$ if $0<a \leq a_1$ 
by \eqref{Eq_403} and \eqref{Eq_404}. 
Therefore, 
\[
\aligned 
\frac{1}{2\pi}  \int_{-\infty + ic}^{\infty + ic} 
|\Phi_1(\phi; \pm z)|^2
 \, dz 
& \leq C_2 
\int_{-a}^{a} \int_{-a}^{a} \phi(u)\overline{\phi(t)}\,K(t,u)
\, dudt  \\ 
& = C_2 \cdot \frac{1}{2\pi}  \int_{-\infty}^{\infty} 
|\Phi_1(\phi; \pm z)|^2
 \, dz. 
\endaligned 
\]
This leads to 
\[
I_0+I_1
 \leq 2 \, C_1C_2 \cdot 
\frac{1}{2\pi}  \int_{-\infty}^{\infty} 
|\Phi_1(\phi;z)|^2
 \, dz 
+ 2 \, C_1C_2 \cdot 
\frac{1}{2\pi}  \int_{-\infty}^{\infty} 
|\Phi_1(\phi;-z)|^2
 \, dz .
\]

Following the above, we take $C>0$ such that $C > 3C_1C_2$. 
Since 
\begin{equation} \label{Eq_410}
\Re\left[
\frac{\Gamma'}{\Gamma}(\sigma+it) 
\right] = \log |\sigma+it| + O(|\sigma+it|^{-1})
\end{equation}
for a fixed $\sigma$, we can fine $t_0>0$ 
so that 
\[
\Re\left[
\frac{\Gamma'}{\Gamma}\left(\frac{1}{4}-\frac{iz}{2} \right) 
- \frac{1}{2} \log \pi
\right] \geq C
\]
if $z \in \R$ and $|z| \geq t_0$.
We put 
\[
C_0 = \max_{|z| \leq t_0}\Re\left[
\frac{\Gamma'}{\Gamma}\left(\frac{1}{4}-\frac{iz}{2} \right) 
- \frac{1}{2} \log \pi
\right]. 
\]
Then we have 
\[
\aligned 
\frac{1}{2\pi} &\int_{-\infty}^{\infty} 
\Re\left[
\frac{1}{2}\frac{\Gamma'}{\Gamma}\left(\frac{s}{2}\right)- \frac{1}{2}\log \pi 
 \right]
|\Phi_1(\phi;z)|^2
 \, dz \\
& \geq C \cdot 
\frac{1}{2\pi} \int_{|z| \geq t_0}
|\Phi_1(\phi;z)|^2
 \, dz
-
C_0 \cdot 
\frac{1}{2\pi} \int_{|z| \leq t_0}
|\Phi_1(\phi;z)|^2
 \, dz \\
& \geq C \cdot 
\frac{1}{2\pi} \int_{-\infty}^{\infty} 
|\Phi_1(\phi;z)|^2
 \, dz
-
(C+C_0) \cdot 
\frac{1}{2\pi} \int_{|z| \leq t_0}
|\Phi_1(\phi;z)|^2
 \, dz.
\endaligned 
\]

From the above calculations for $I_0$, $I_1$, and $I_\infty$, 
we obtain 
\begin{equation} \label{Eq_411}  
\aligned 
\langle & \phi, \phi \rangle_{G_g,a} \\
& \geq 
(C-2C_1C_2) \frac{1}{2\pi} \int_{-\infty}^{\infty} 
|\Phi_1(\phi;z)|^2
 \, dz-
(C+C_0) \cdot 
\frac{1}{2\pi} \int_{|z| \leq t_0}
|\Phi_1(\phi;z)|^2
 \, dz \\
&  +
(C-2C_1C_2) \frac{1}{2\pi} \int_{-\infty}^{\infty} 
|\Phi_1(\phi;-z)|^2
 \, dz-
(C+C_0) \cdot 
\frac{1}{2\pi} \int_{|z| \leq t_0}
|\Phi_1(\phi;-z)|^2
 \, dz.
\endaligned 
\end{equation}

For the second and fourth integrals on the right-hand side, 
\[
\aligned 
\frac{1}{2\pi} & \int_{|z| \leq t_0} 
|\Phi_1(\phi;z)|^2
 \, dz 
= 
\int_{-a}^{a} \int_{-a}^{a} \phi(u)\overline{\phi(t)}
\left( \frac{1}{2\pi} 
\int_{|z| \leq t_0} \frac{e^{iuz}-1}{z}\frac{e^{-itz}-1}{z} \, dz \right) 
\, dudt.
\endaligned 
\]
If $0<a<a_1$ is sufficiently small, we have 
\[
\frac{1}{2}
\left( 
\frac{e^{iuz}-1}{z}\frac{e^{-itz}-1}{z} 
+
\frac{e^{-iuz}-1}{z}\frac{e^{itz}-1}{z} 
\right)
= ut\, ( 1 + O(|z|) ) \ll a K(t,u)
\]
for $z \in \R$ with $|z| \leq t_0$, 
since both $|uz|$ and $|tz|$ are small in the range $|z| \leq t_0$. 
Therefore, there exists $0<a_2<a_1$ and $C_3>0$ such that  
\[
\frac{1}{2\pi} 
\int_{|z| \leq t_0} \frac{e^{iuz}-1}{z}\frac{e^{-itz}-1}{z} \, dz  \leq a \cdot C_3 \cdot K(t,u)
\]
for every $0<a \leq a_2$. Hence, 
\[
\frac{1}{2\pi} \int_{|z| \leq t_0} |\Phi_1(\phi;z)|^2 \, dz
\leq a \cdot C_3 \cdot 
\frac{1}{2\pi} \int_{-\infty}^{\infty} |\Phi_1(\phi;z)|^2 \, dz. 
\]
Applying this to \eqref{Eq_411}, we obtain 
\[
\aligned 
\langle \phi, \phi \rangle_{G_g,a}
& \geq 
((C-2C_1C_2)-a(C+C_0)C_3) \frac{1}{2\pi} \int_{-\infty}^{\infty} 
|\Phi_1(\phi;z)|^2 \, dz\\
& \quad +
((C-2C_1C_2)-a(C+C_0)C_3) \frac{1}{2\pi} \int_{-\infty}^{\infty} 
|\Phi_1(\phi;-z)|^2 \, dz
\endaligned 
\]
for $0<a<a_2$. 
The coefficients of the integrals on the right-hand side 
are positive if $0<a<a_2$ is sufficiently small. 
\end{proof}

To state another analog of Yoshida's result, 
we define 
\begin{equation*} 
\mathfrak{K}_{N,0}(a) :=\left\{ \phi \in L^2(-a,a)\,\left|~\widehat{\phi}(0)=0,~
~I_0^{(a)}(\phi) \in K_N(a) \right.\right\},
\end{equation*}
where $K_N(a)$ is the space in \eqref{Eq_401}. 
The following is also used to prove Theorem \ref{Thm_1_4}. 

\begin{theorem} \label{Thm_4_3} 
Let $a_0>0$ and $\mu>0$ be given numbers. 
Then there exists $N \geq 0$ such that 
\begin{equation} \label{Eq_412} 
\langle \phi, \phi \rangle_{G_g,a} \geq \mu \int_{-\infty}^{\infty} |\Phi_1(\phi,z)|^2 \, dz
\end{equation}
for every $\phi \in \mathfrak{K}_{N,0}(a)$ and $0<a \leq a_0$, 
where $\Phi_1(\phi,z)$ is the transform of $\phi$ in \eqref{Eq_402}.  
\end{theorem}
\begin{proof}
By taking $a_1$ larger than $a_0$ 
and taking $C$ as $C > 3C_1C_2+\mu$ in the proof of Theorem \ref{Thm_4_2}, 
we obtain the lower bound \eqref{Eq_411} for $0<a \leq a_0$ 
with $C-2C_1C_2 > \mu + C_1C_2$. 
Therefore, the proof is complete if the estimate 
\begin{equation} \label{Eq_413}  
\int_{|z| \leq t_0}
|\Phi_1(\phi;z)|^2
 \, dz 
\ll N^{-1/2}
\int_{-\infty}^{\infty}
|\Phi_1(\phi;z)|^2
 \, dz
\end{equation}
is proven for $\phi \in \mathfrak{K}_{N,0}(a)$ and $0<a \leq a_0$, 
where the implied constant depends only on $a_0$. 
This is shown in the same way as \cite[Lemma 3]{Yo92} as follows. 

Let $\phi \in \mathfrak{K}_{N,0}(a)$ and put $\psi = I_0^{(a)}(\phi)~(\in K_N(a))$. 
Then $\psi(-a)=0$, $\psi(a)=\widehat{\phi}(0)=0$,
$\Phi_1(\phi,z)
= -i \widehat{\psi}(z)$, and 
\[
\aligned 
\int_{-\infty}^{\infty} |\Phi_1(\phi,z)|^2 \,dz 
= \int_{-\infty}^{\infty} |\widehat{\psi}(z)|^2 \,dz 
= \Vert \psi \Vert_{L^2}^2
\endaligned 
\]
by the Parseval identity. 
Let $\psi(t)=(2a)^{-1/2}\sum_{|n|>N}c_n e^{\pi i n t/a}$ 
be the Fourier expansion of $\psi$ in $L^2(-a,a)$. 
Then $\Vert \psi \Vert_{L^2}^2 = \sum_{|n|>N} |c_n|^2$, $c_n \ll n^{-k}$ 
for any $k \in \Z_{\geq 0}$,  and 
\[
\aligned 
\widehat{\psi}(z) 
& = \frac{1}{\sqrt{2a}} 
\left[ 
\sum_{|n|>N}c_n \frac{a}{\pi i n} e^{\pi i n t/a} \, e^{izt}
\right]_{t=-a}^{a} 
 - \frac{1}{\sqrt{2a}} 
\sum_{|n|>N}c_n \frac{a}{\pi i n} e^{\pi i n t/a} \, iz e^{izt} 
\, dt
\endaligned 
\]
by applying integration by parts to the Fourier integral. Hence we obtain
\[
\widehat{\psi}(t) 
\leq \sqrt{2a} (1+a|z|) \sum_{|n|>N} \left|\frac{c_n}{\pi n}\right|
\leq \sqrt{2a} (1+a|z|) \left( \sum_{|n|>N} \frac{1}{(\pi n)^2} \right)^{1/2} \Vert \psi \Vert_{L^2}
\]
for $z \in \R$. For the second inequality, we used the Schwartz inequality.  
This inequality implies \eqref{Eq_413}, since $\Phi_1(\phi,z)=-i\widehat{\psi}(z)$.    
\end{proof}

Theorem \ref{Thm_4_3} can also be proved using \eqref{Eq_308} 
and \cite[Lemma 3]{Yo92}, but here we performed a direct proof as above.

%
\section{Proof of Theorem \ref{Thm_1_4} } \label{Section_5}
%

First, we show the second half of Theorem \ref{Thm_1_4}.
Since $g(t)$ is continuous and hermitian on the real line, 
the integral operator $\mathsf{G}_g[a]$ on $L^2(-a,a)$ 
defined in \eqref{eq_110} is a self-adjoint Hilbert--Schmidt operator for every $0<a<\infty$. 
Therefore, the spectrum of $\mathsf{G}_g[a]$ consists of eigenvalues and zero.  
Whether or not zero is an eigenvalue is not determined by the general theory. 
The hermitian form $\langle \cdot, \cdot \rangle_{G_g,a}$ 
is non-negative definite 
if and only if all eigenvalues of $\mathsf{G}_g[a]$ are non-negative,  
and the nondegeneracy for $\langle \cdot, \cdot \rangle_{G_g,a}$ 
is equivalent to that $\mathsf{G}_g[a]$ does not have zero as an eigenvalue, 
since 
$\langle \phi_1, \phi_2 \rangle_{G_g,a}=\langle \mathsf{G}_g[a] \phi_1, \phi_2\rangle_{L^2}$ 
and there exists an orthonormal basis for $L^2(-a,a)$ 
consisting of eigenfunctions of $\mathsf{G}_g[a]$. 

\subsection{Proof of necessity}  \label{Section_5_1}
If the RH is true, $G_g(t,u)$ is non-negative definite 
on $(-a,a)$ for every $0< a < \infty$ by Theorem \ref{Thm_1_2}. 
Therefore, all eigenvalues $\mathsf{G}_g[a]$ are non-negative (\cite[\S8]{St76}). 
If there exists an eigenfunction for eigenvalue zero, 
it must be zero as a function by \eqref{Eq_301}, \eqref{Eq_302}, and Lemma \ref{Lem_2_1}. 
This is a contradiction. 
Hence, $\langle \cdot, \cdot \rangle_{G_g,a}$ is nondegenerate. 
\hfill $\Box$

\subsection{Proof of sufficiency} \label{Section_3_2}
We prove the sufficiency with the same strategy as in the proof of \cite[Theorem 2]{Yo92}. 
That is, we prove that $\langle \cdot, \cdot \rangle_{G_g,a}$  
degenerates on $L^2(-a,a)$ at $a=a_0$ for some $a_0>0$ if the RH is false. 

Let  $A$ be the set of all positive real numbers $a$ 
such that $\langle \cdot, \cdot \rangle_{G_g,a}$ is non-negative definite on $L^2(-a,a)$
and set $B=(0,\infty)\setminus A$. 
If $b \in B$, there exists $\phi \in L^2(-b,b)$ such that $\langle \phi, \phi \rangle_b<0$. 
Then $\langle \phi, \phi \rangle_{b'}<0$ if $b'$ is sufficiently close to $b$ 
by the continuity of the kernel $G_g(t,u)$. 
Hence $B$ is open and $A$ is a closed subset of $(0,\infty)$. 
Assume that the RH is false. 
Then $A$ is bounded by Theorem \ref{Thm_1_3}. 
Therefore, we prove that $\langle \cdot, \cdot \rangle_{G_g,a}$ 
degenerates on $L^2(-a,a)$  at $a=a_0$ 
for the maximum element $a_0$ of $A$. 
\bigskip

Let $a_1$ be a real number greater than $a_0$ 
and let $\mu>0$. 
For $n \in \Z$, we put
\[
\chi_n[a](t) := \frac{1}{\sqrt{2a}} \exp\left( \frac{\pi i n t}{a} \right).
\]
Then $\{\chi_{n}[a]\}_{n \in \Z}$ forms 
an orthonormal basis of $L^2(-a,a)$. 
We set  
\[
\chi_{n,m}[a](t):=
\frac{\pi i n}{a}
\chi_{n}[a](t) -  (-1)^{n+m}\frac{\pi i m}{a}\chi_{m}[a](t) 
\]
for $m,n\not=0$. 
Then $\widehat{\chi_{n,m}[a]}(0)=0$ and integrals 
$I_0^{(a)}(\chi_{n,n+1}[a])$ and $I_0^{(a)}(\chi_{-n,-(n+1)}[a])$ 
belong to $K_N(a)$ for every $n>N \in \N$, 
because 
\[
I_0^{(a)}(\chi_{n,m}[a])
= \chi_{n}[a] - (-1)^{m+n} \chi_{m}[a].
\]

We denote by $\mathcal{V}_N(a)$ the closed subspace of $L^2(-a,a)$ 
spanned by $\chi_{n,n+1}[a]$ and $\chi_{-n,-(n+1)}[a]$ 
for all $n>N$. Based on Theorem \ref{Thm_4_3}, 
we can take $N$ such that 
\begin{equation} \label{Eq_501}
\langle v,v \rangle_{G_g,a} \geq \mu \Vert I_0^{(a)}(v) \Vert_{L^2(-a,a)}
\end{equation}
for every $v \in \mathcal{V}_N(a)$ and $0<a \leq a_1$,  
because inequality \eqref{Eq_412}  extends to $v$ in the $L^2$-closure 
$\mathcal{V}_N(a)$ of a subspace of $\mathfrak{K}_{N,0}(a)$ 
and $\Phi_1(v,\cdot)=\widehat{I_0^{(a)}(v)}$. 
In particular $\langle v, v \rangle_{G_g,a}>0$ for every non-zero $v \in \mathcal{V}_N(a)$. 
\bigskip

Let $\mathcal{V}_N(a)^\perp$ be the orthogonal complement 
of $\mathcal{V}_N(a)$ in $L^2(-a,a)$ 
with respect to $\langle\cdot,\cdot\rangle_{L^2(-a,a)}$. 
We put
\begin{equation} \label{Eq_502}
\aligned 
u_n[a] & := \chi_n[a] \quad \text{if $|n| \leq N$}, \\
u_{\pm(N+1)}[a] & := \sum_{k=1}^{\infty} \frac{(-1)^k}{N+k} \cdot \chi_{\pm(N+k)} \\
& =  \frac{(-1)^{N+1}}{\sqrt{2a}}\left[
\log\left(1+e^{\pm \frac{\pi i t}{a}}\right)
+ \sum_{k=1}^{N}\frac{(-1)^k}{k} e^{\pm\frac{\pi i k t}{a}}
\right], \\
u_{\pm n}[a] & := \chi_{\pm n, \pm(n+1)}[a] \quad \text{if $n \geq N+1$}. 
\endaligned 
\end{equation}
Then 
$u_{\pm(N+1)}[a]$ are orthogonal to $u_n[a]$ ($|n|\not=N+1$), 
$u_{N+1}[a]$ and $u_{-(N+1)}[a]$ are orthogonal, 
and 
$\widehat{u_{\pm(N+1)}[a]}(0)=0$. 
Further the space $\mathcal{V}_N(a)^\perp$ is a $(2N+3)$ dimensional space 
spanned by $\{u_{n}[a]\}_{|n|\leq N+1}$. 
As we will show later, 
for each $u_n[a]$ ($|n| \leq N+1$), 
there exists $v_n[a]\in \mathcal{V}_N(a)$ such that 
\begin{equation} \label{Eq_503}
\langle v, u_n[a] \rangle_{G_g,a} = \langle v, v_n[a] \rangle_{G_g,a}
\quad \text{for every $v \in \mathcal{V}_N(a)$}. 
\end{equation}
Each $u_n[a]-v_n[a]$ is orthogonal 
to $\mathcal{V}_N(a)$ for $\langle \cdot, \cdot \rangle_{G_g,a}$.  
Let 
\[
\mathcal{A}_N(a)
= \{ \phi \in L^2(-a,a)~|~\langle \phi, v \rangle_{G_g,a}=0 
~\text{for all $v \in \mathcal{V}_N(a)$}\}
\]
be the annihilator of  $\mathcal{V}_N(a)$ for $\langle \cdot, \cdot \rangle_{G_g,a}$. 
Then $\langle \cdot, \cdot \rangle_{G_g,a}$ 
is non-negative definite on $L^2(-a,a)$ 
if and only if it is non-negative definite on $\mathcal{A}_N(a)$, 
since any element of $L^2(-a,a)$ can be written as 
a linear combination of  $u_n[a]-v_n[a] \in \mathcal{A}_N(a)$ 
($|n| \leq N+1$) and $v \in \mathcal{V}_N(a)$. 

By definition of $a_0$, 
$\langle \cdot, \cdot \rangle_{G_g,a}$ is non-negative definite 
on $\mathcal{A}_N(a)$ for $0<a \leq a_0$ 
and is not non-negative definite for $a_0 < a \leq a_1$. 
The hermitian form $\langle \cdot, \cdot \rangle_{G_g,a}$ on $\mathcal{A}_N(a)$ 
is represented by matrix coefficients
\begin{equation} \label{Eq_504}
\langle u_m[a]-v_m[a], u_n[a]-v_n[a]\rangle_{G_g,a}, \quad |m|,|n| \leq N+1.
\end{equation}
By \eqref{Eq_503}, we have 
\[
\langle u_m[a]-v_m[a], u_n[a]-v_n[a]\rangle_{G_g,a}
= \langle u_m[a], u_n[a] \rangle_{G_g,a} 
- \langle v_m[a], v_n[a]\rangle_{G_g,a}. 
\]
The continuity of 
$\langle u_m[a], u_n[a] \rangle_{G_g,a}$ 
is trivial. Therefore, 
the proof of Theorem \ref{Thm_1_4} is reduced to 
the continuity of $\langle v_m[a], v_n[a] \rangle_{G_g,a}$ 
in a neighborhood $U_{a_0}$ of $a=a_0$ 
as in the proof of \cite[Theorem 2]{Yo92}, 
since such fact implies 
that $\langle \cdot, \cdot \rangle_{G_g,a}$ 
degenerates on $L^2(-a,a)$ at $a=a_0$. 
\bigskip

We show the existence of an orthogonal projection 
$v \in \mathcal{V}_N(a)$ of $u \in \mathcal{V}_N(a)^\perp$ 
with respect to $\langle \cdot,\cdot \rangle_{G_g,a}$. 
Put $d={\rm dim}\,\mathcal{V}_N(a)^\perp=2N+3$. 
Take a basis $\{\phi_k\}_{k=1}^{d}$ of $\mathcal{V}_N(a)^\perp$ 
and denote $\psi_{k}$ the orthogonal projection of $\phi_k$ 
in the closed subspace $\overline{\mathsf{G}_g[a](\mathcal{V}_N(a))}$ of $L^2(-a,a)$ 
with respect to $\langle \cdot,\cdot \rangle_{L^2(-a,a)}$. 
Then each $\phi_k-\psi_{k}$ belongs to $\mathcal{A}_N(a)$, 
since $\langle \phi_k-\psi_{k}, v\rangle_{G_g,a}  
= \langle \phi_k-\psi_{k}, \mathsf{G}_g[a]v \rangle_{L^2(-a,a)}=0$. 
We decompose it as $\phi_k-\psi_{k} =u_{k}-v_{k}$ 
by $u_{k} \in \mathcal{V}_N(a)^\perp$ and $v_{k} \in \mathcal{V}_N(a)$ 
according to the direct sum $L^2(-a,a)=\mathcal{V}_N(a)^\perp \oplus \mathcal{V}_N(a)$. 
Then $\{u_{k}\}_{k=1}^{d}$ is also a basis of $\mathcal{V}_N(a)^\perp$, 
because, if $\{u_{k}\}_{k=1}^{d}$ is linearly dependent, 
a non-zero linear combination of $v_{k}$'s must belongs to $\mathcal{A}_N(a) \cap \mathcal{V}_N(a)$ 
which is $\{0\}$ by the positivity of $\langle \cdot, \cdot \rangle_{G_g,a}$ on $\mathcal{V}_N(a)$. 
As a result, we obtain the direct sum 
$L^2(-a,a) = \mathcal{A}_N(a) \oplus \mathcal{V}_N(a)$ 
which is also orthogonal for $\langle \cdot, \cdot \rangle_{G_g,a}$, 
that is, for each $u \in \mathcal{V}_N(a)^\perp$ 
there exists $v \in \mathcal{V}_N(a)$ such that $u - v \in \mathcal{A}_N(a)$. 
\bigskip

We back to the proof of the continuity of $\langle v_m[a], v_n[a] \rangle_{G_g,a}$. 
Hereafter in this proof, we abbreviate 
$\chi_n[a]$ as $\chi_n$, 
$u_n[a]$ as $u_n$,  $v_n[a]$ as $v_n$,  
$\langle \cdot,\cdot \rangle_{G_g,a}$ as 
$\langle \cdot,\cdot \rangle$, 
$I_0^{(a)}(\phi)$ as $I(\phi)$, and 
$\langle \cdot,\cdot \rangle_{L^2(-a,a)}$ as 
$\langle \cdot,\cdot \rangle_{L^2}$. 
Let $\{\tau_k^+\}$ and $\{\tau_k^-\}$ be orthonormal systems obtained from 
$\{I(u_{n})\}_{n \geq N+1}
=\{\chi_{n}+\chi_{n+1}\}_{n \geq N+1}$
and 
$\{I(u_{-n})\}_{n \geq N+1}
=\{\chi_{-n}+\chi_{-(n+1)}\}_{n \geq N+1}$ 
by the Gram--Schmidt orthogonalization process 
with respect to $\langle \cdot,\cdot \rangle_{L^2}$, 
respectively. We have
\begin{equation} \label{Eq_505}
\tau_j^\pm= \frac{(-1)^{j-1}}{\sqrt{j(j+1)}} 
\left( \sum_{k=1}^{j} (-1)^{k-1} \chi_{\pm(N+k)} -j(-1)^{(j+1)-1}\chi_{\pm(N+j+1)} 
\right)
\end{equation}
by the determinant formula 
\begin{equation} \label{Eq_506}
\tau_j^\pm 
= \frac{1}{\sqrt{G_{j-1}G_j}}
\begin{vmatrix}
\langle \omega_1, \omega_1 \rangle_{L^2} & \langle \omega_1, \omega_2 \rangle_{L^2} & \cdots & \langle \omega_1, \omega_{j-1} \rangle_{L^2} &  \omega_1 \\
\langle \omega_2, \omega_1 \rangle_{L^2} & \langle \omega_2, \omega_2 \rangle_{L^2} & \cdots & \langle \omega_2, \omega_{j-1} \rangle_{L^2} & \omega_2 \\
\vdots & \vdots & \ddots & \vdots & \vdots \\
\langle \omega_j, \omega_1 \rangle_{L^2} & \langle \omega_j, \omega_2\rangle_{L^2} & \cdots & \langle \omega_j, \omega_{j-1} \rangle_{L^2}  & \omega_j \
\end{vmatrix}, 
\end{equation}
where $\omega_k=\chi_{\pm(N+k)}+\chi_{\pm(N+k+1)}$, 
$G_0=1$, and 
\begin{equation} \label{Eq_507}
G_j
= 
\begin{vmatrix}
\langle \omega_1, \omega_1 \rangle_{L^2} & \langle \omega_1, \omega_2 \rangle_{L^2} & \cdots & \langle \omega_1, \omega_j \rangle_{L^2} \\
\langle \omega_2, \omega_1 \rangle_{L^2} & \langle \omega_2, \omega_2 \rangle_{L^2} & \cdots & \langle \omega_2, \omega_j \rangle_{L^2} \\
\vdots & \vdots & \ddots & \vdots \\
\langle \omega_j. \omega_1 \rangle_{L^2} & \langle \omega_j, \omega_2 \rangle_{L^2} & \cdots & \langle \omega_j, \omega_j \rangle_{L^2} \\
\end{vmatrix}\,(=j+1).
\end{equation}
We write $\tau_{2n-1}=\tau_n^+$ and $\tau_{2n}=\tau_n^-$ 
and define the basis $\{\eta_j\}$ of $\mathcal{V}_N(a)$ by
\[
\aligned 
\eta_j = D(\tau_j), \quad j=1,2,\cdots.
\endaligned 
\]  
Then, there exists a sequence of positive numbers $\mu(M)$, $M \in \N$ 
independent of $a \in U_{a_0}$ 
such that $\lim_{M \to \infty} \mu(M) = \infty$ and that
\begin{equation} \label{Eq_508}
\langle v,v \rangle \geq \mu(M) \Vert I(v) \Vert_{L^2}
\end{equation}
if $v \in \mathcal{V}_N(a)$ satisfies $\langle I(v), I(\eta_i) \rangle_{L^2}
=\langle I(v), \tau_i \rangle_{L^2}=0$ 
for every $i \leq M$ when $a \in U_{a_0}$ 
by the orthogonality of $\{\tau_k\}$ for $\langle\cdot,\cdot \rangle_{L^2}$
and Theorem \ref{Thm_4_3}. 

We have 
$\langle \eta_j , \eta_l \rangle=\langle D(\tau_j) , D(\tau_l) \rangle
=\langle \tau_j , \tau_l \rangle_W$,
\[ 
\langle u_n , \eta_j \rangle
=
\langle \chi_n , \eta_j \rangle
=\frac{a}{\pi i n}\langle \chi_n-(-1)^n \chi_0 , \tau_j \rangle_W \quad (0<|n| \leq N)
\]
\[ 
\langle u_{\pm(N+1)} , \eta_j \rangle
=
\frac{a}{\pi i}
\sum_{k=1}^{\infty} \frac{(-1)^k}{(N+k)^2} 
\langle (\chi_{\pm(N+k)}-(-1)^{N+k}\chi_0), \tau_j \rangle_W
\]
by $\chi_n = D(a(\pi i n)^{-1} (\chi_n-(-1)^n \chi_0))$,  
\[
u_{\pm(N+1)}
=
D\left[ \frac{a}{\pi i}\sum_{k=1}^{\infty}  \frac{(-1)^k}{(N+k)^2} 
(\chi_{N+k}-(-1)^{N+k}\chi_0)
\right], 
\]
$\chi_k-(-1)^k \chi_0 \in C(a)$ ($k \in \Z$), 
the extension of \eqref{Eq_308} to $\mathcal{V}_N(a)$, and 
the definition $\langle u,v \rangle_W = W(u \ast \widetilde{v})$. 
Therefore, by definition \eqref{Eq_502}, 
estimates of $\langle \eta_j , \eta_l \rangle$ 
and $\langle u_n , \eta_j \rangle$ ($0<|n| \leq N+1$) are reduced 
to the estimates of  $\langle \chi_m, \chi_n \rangle_W$ ($m,n \in \Z$), 
but they are known as 
$|\langle \chi_m,\chi_n \rangle_W| \ll |m-n|^{-1}$ 
if $m\not=n$ and 
$|\langle \chi_n,\chi_n \rangle_W| \sim \log |n|$ 
as $|n| \to \infty$ 
in \cite[(5.17) and (5.18)]{Yo92}, 
where the implied constants can be taken independent of $a \in U_{a_0}$. 
From these estimates and the explicit formula \eqref{Eq_505}, 
we can easily deduce the estimates 
\begin{equation} \label{Eq_509}
|\langle \eta_j , \eta_l \rangle| \ll 
 \frac{1}{|j-l|} + \frac{1}{\max\{j,l\}} \log \max\{j,l\} 
\end{equation}
for $j, l \in \N$ with $j \not= l$ and 
\begin{equation} \label{Eq_510}
|\langle u_n, \eta_j \rangle|=
|\langle \chi_n, \eta_j \rangle| \ll 
\frac{\log j}{j} ~ (0<|n| \leq N), 
\qquad |\langle u_{\pm (N+1)}, \eta_j \rangle| \ll \frac{\log j}{j}
\end{equation}
for $j \in \N$, where the implied constants independent of $a \in U_{a_0}$. 
Further, we obtain
\begin{equation} \label{Eq_511}
|\langle u_0, \eta_j \rangle|=
|\langle \chi_0, \eta_j \rangle| \ll 
\frac{(\log j)^2}{j} 
\end{equation}
for $j \in \N$. This estimate cannot be reduced 
to the estimate of $\langle \chi_m,\chi_n \rangle_W$ 
and must be treated separately. 
But we leave that for later (Section \ref{Section_3_3} below) 
and continue with the proof of Theorem \ref{Thm_1_4}.
\medskip

Let $\{\psi_k\}$ be the basis of $\mathcal{V}_N(a)$ obtained from  $\{\eta_k\}$ 
by the Gram--Schmidt orthogonalization process with respect to $\langle \cdot,\cdot \rangle$. 
The positivity of $\langle \cdot,\cdot \rangle$ on $\mathcal{V}_N(a)$ 
ensures that this orthogonalization process works. 
By \eqref{Eq_503}, we have
\[
v_n = \sum_{i=1}^{\infty} f_{ni} \psi_i, \quad
 f_{ni}=\langle v_n, \psi_i \rangle=\langle u_n, \psi_i \rangle \quad (|n| \leq N+1). 
\]
Then $f_{ni}$ is a continuous function of $a$ 
by \eqref{Eq_502}, \eqref{Eq_506}, \eqref{Eq_507}, 
and the obvious continuity of $\langle \chi_m,\chi_n \rangle$ for $a$. 
We have 
\[
\langle v_m, v_n \rangle = \sum_{i=1}^{\infty} f_{mi} \overline{f_{ni}}, 
\quad 
\left( \sum_{i>M} |f_{mi} \overline{f_{ni}}| \right)^2 \leq 
\sum_{i>M} |f_{mi}|^2  \sum_{i>M} |f_{ni}|^2.
\]
Therefore the continuity of $\langle v_m, v_n \rangle$ 
for $a$ reduces to the uniformity of convergence 
of $\sum_{i=1}^{\infty} |f_{ni}|^2=\sum_{i=1}^{\infty} |\langle u_n, \psi_i \rangle|^2$ 
in $U_{a_0}$ for every $|n| \leq N+1$. 

To prove such convergence, we write
\[
\psi_i = \sum_{j=1}^{i} d_{ij} \eta_j, \qquad 
\eta_i = \sum_{j=1}^{i} c_{ij} \psi_j
\]
by infinite dimensional lower triangular matrices $C = (c_{ij})$ and $D = (d_{ij})$. 
For a positive integer $M$, we set
\[
C = 
\begin{bmatrix} 
X_1 & 0 \\ X_3 & X_4
\end{bmatrix}, \quad 
D = 
\begin{bmatrix} 
Y_1 & 0 \\ Y_3 & Y_4
\end{bmatrix}, 
\]
where $X_1$ and $Y_1$ denote the first $M \times M$-blocks. 
Let $\Vert A \Vert$ be the operator norm of 
the operator $A$ on a Hilbert space $\ell^2$ of row vectors   
defined by $\Vert A \Vert = \sup_{\Vert x \Vert_{\ell^2}=1} \Vert x A \Vert_{\ell^2}$. 
Using the estimates \eqref{Eq_501}, \eqref{Eq_508}, 
and 
\eqref{Eq_509}--\eqref{Eq_511}, 
we obtain 
\begin{equation} \label{Eq_512}
\Vert X_1^{-1} \Vert \ll 1, \quad 
\Vert X_3 \Vert \ll 1, \quad 
\Vert X_4^{-1} \Vert \ll (\mu(M)+O(1))^{-1}
\end{equation}
with implied constants independent of $a \in U_{a_0}$ 
in the same way as the proof of \cite[Lemma 9]{Yo92}, 
because the difference between \eqref{Eq_509} and (II) of \cite[Lemma 9]{Yo92} 
and the difference between \eqref{Eq_510}--\eqref{Eq_511} and (III) of \cite[Lemma 9]{Yo92} 
do not affect the calculations to prove these boundedness.
\medskip

Finally, we estimate $\sum_{i>M}|\langle u_n, \psi_i \rangle|^2$ for $|n| \leq N+1$. 
Since 
\[\sum_{i>M}|\langle u_n,  \psi_i \rangle|
=\sum_{i>M}\left|\sum_{j=1}^{i} d_{ij} \langle u_n,  \eta_j \rangle\right|^2,
\] 
we get 
\[
\sum_{i>M} |\langle u_n,  \psi_i  \rangle|^2
= \left\Vert \xi \begin{bmatrix} {}^{t}Y_3 \\ {}^{t} Y_4 \end{bmatrix} \right\Vert_{\ell^2}^2,  
\]
where $\xi=(\langle u_n, \eta_1 \rangle,~\langle u_n,  \eta_2 \rangle,\cdots,
\langle u_n,  \eta_j \rangle,\cdots)$. By \eqref{Eq_510}--\eqref{Eq_511}, 
\begin{equation} \label{Eq_513}
\Vert \xi \Vert_{\ell^2} \ll 1. 
\end{equation}
From $CD=1$, we get $Y_4=X_4^{-1}$ , $Y_3=-X_4^{-1}X_3X_1^{-1}$. 
Therefore, 
\begin{equation} \label{Eq_514}
\sum_{i>M} |\langle u_n,  \psi_i \rangle|^2
= \left\Vert \xi \begin{bmatrix} {}^{t}X_1^{-1}X_3 \\ 1 \end{bmatrix} {}^{t}X_4^{-1}\right\Vert_{\ell^2}^2. 
\end{equation}
By \eqref{Eq_508}, we obtain 
\[
\left\Vert \xi \begin{bmatrix} {}^{t}X_1^{-1}X_3 \\ 1 \end{bmatrix} {}^{t}X_4^{-1}\right\Vert_{\ell^2}^2 
\ll (\mu(M)+O(1))^{-1} \Vert \xi \Vert_{\ell^2}
\]
with implied constants independent of $a \in U_{a_0}$. 
From \eqref{Eq_513} and \eqref{Eq_514}, 
$\sum_{i=1}^{\infty} |\langle u_n, \psi_i \rangle|^2$ converges uniformly on $U_{a_0}$, 
so the proof is complete. \hfill $\Box$

\subsection{Proof of \eqref{Eq_511}} \label{Section_3_3}

To complete the proof of Theorem \ref{Thm_1_4}, 
we prove \eqref{Eq_511} 
using the same notation as in the second half of Section \ref{Section_3_2}. 
We use the Weil explicit formula in the following form: 
\begin{equation} \label{Eq_515}
\aligned 
\lim_{X \to \infty} & \sum_{|\gamma|\leq X} \widehat{\phi}(\gamma) \\
& = \widehat{\phi}(i/2)+\widehat{\phi}(-i/2) 
 - \sum_{n=1}^{\infty} \frac{\Lambda(n)}{\sqrt{n}} \phi(\log n) 
- \sum_{n=1}^{\infty} \frac{\Lambda(n)}{\sqrt{n}} \phi(-\log n)  \\
& \quad 
- (\log \pi)\phi(0)
+
\frac{1}{2\pi} \int_{-\infty}^{\infty} \Re\left[\frac{\Gamma'}{\Gamma}\left(\frac{1}{4}+\frac{iz}{2}\right) \right]
\widehat{\phi}(z)\, dz, 
\endaligned 
\end{equation}
where the sum on the left-hand side ranges over all zeros $\gamma$ of $\xi(1/2-iz)$ 
counting with multiplicity, 
that is, $\rho=1/2-i\gamma$ for complex zeros $\rho$ of the Riemann zeta-function. 
Formula \eqref{Eq_515} is obtained 
from the explicit formula in \cite[p. 186]{Bo01}  
by taking $\phi(t) = e^{t/2}f(e^t)$ for $f(x)$ in that formula 
and noting the symmetry of zeros $\gamma \mapsto -\gamma$ .  
For the conditions for test functions, we use the conditions in \cite[Section 3]{BoLa99}.

We prove \eqref{Eq_511} by calculating $\langle \chi_0,\eta_j \rangle$ 
directly using \eqref{Eq_515} as in \cite[(5.15)--(5.18)]{Yo92}. 
However, we only give the outlines, 
since the argument of the proof is similar to \cite{Yo92} 
except for the choice of test functions $\phi(t)$.
\medskip

By integrating by parts using \eqref{Eq_109} and \eqref{Eq_110},
\begin{equation} \label{Eq_516}
\aligned 
\langle \chi_0, \chi_k \rangle
& = \sum_\gamma \frac{(\cos(a\gamma)-1)}{\gamma^2} \cdot 
 \frac{(-1)^k2a\sin(a\gamma)}{k\pi+a\gamma} \\
& \quad 
 -  \sum_\gamma \frac{a\gamma\cos(a\gamma)-\sin(a\gamma)}{a \gamma^3} \cdot 
 \frac{(-1)^k2a\sin(a\gamma)}{k\pi+a\gamma} .
\endaligned 
\end{equation}
The functions on the right-hand side 
are expressed as Fourier transforms as follows: 
\[
\aligned 
\int_{-a}^{a} \frac{1}{2} (|t|-a)e^{iz t} \, dt & = \frac{\cos(az)-1}{z^2}, \\
\int_{-a}^{a} \frac{1}{4a} (t^2-a^2)e^{iz t} \, dt 
& = \frac{az\cos(az)-\sin(az)}{az^3}, \\
\int_{-a}^{a}  \exp\left( \frac{\pi i k t}{a} \right) e^{iz t} \, dt 
& = \frac{(-1)^k2a\sin(az)}{\pi k + az}. 
\endaligned 
\]
Therefore, if we set $\phi_{1,k}(t)$ and $\phi_{2,k}(t)$ as 
\[
\aligned 
\phi_{1,k}(t):&=\int_{-a}^{a} \exp\left( \frac{\pi i k (t-u)}{a} \right)\mathbf{1}_{(-a,a)}(t-u)
\frac{1}{2} (|u|-a) \, du, \\
\phi_{2,k}(t):&=\int_{-a}^{a} \exp\left( \frac{\pi i k (t-u)}{a} \right)\mathbf{1}_{(-a,a)}(t-u)
\frac{1}{4a} (u^2-a^2) \, du, 
\endaligned 
\]
we have
\[
\aligned 
\widehat{\phi_{1,k}}(z) & = (-1)^k2a \frac{(\cos(az)-1)\sin(az)}{z^2(k\pi+az)}, \\
\quad 
\widehat{\phi_{2,k}}(z) & = (-1)^k2\frac{(az\cos(az)-\sin(az))\sin(az)}{z^3(k\pi+az)}. 
\endaligned 
\]
These functions are compactly supported and continuously differentiable functions 
as follows:  
\[
\phi_{1,k}(t)
 = 
\begin{cases}
~0, & |t| \geq 2a; \\
~\displaystyle{ \frac{a(-1)^k}{2\pi^2k^2}\left[
a \exp\left( \frac{\pi i kt}{a} \right) - a-\pi i k(t-2a) \right]}, & a \leq t \leq 2a; \\[10pt]
~\displaystyle{ \frac{a}{2\pi^2k^2}\left[
a 
((-1)^k-2)\exp\left( \frac{\pi i kt}{a} \right) 
+ (-1)^k\Bigl(a+\pi i kt\Bigr) \right]} & -a < t < a; \\[10pt]
~\displaystyle{ \frac{a(-1)^k}{2\pi^2k^2}\left[ 
a \exp\left( \frac{\pi i kt}{a} \right) - a-\pi i k(t+2a) \right]}, & -2a \leq t \leq -a,
\end{cases}
\]
\[
\phi_{2,k}(t)
 = 
\begin{cases}
~0, \qquad  |t| \geq 2a; &  \\[14pt]
~\displaystyle{ \frac{i(-1)^{k+1}}{4\pi^3k^3}\left[
2a^2 (1+\pi i k)\exp\left( \frac{\pi i kt}{a} \right) 
-2a^2 - 2\pi i a k(t-a) + \pi^2 k^2 t(t-2a)
 \right]}, &  \\[10pt]
\qquad \qquad 0 \leq t \leq 2a;  & \\[14pt]
~\displaystyle{ \frac{i(-1)^k}{4\pi^3k^3}\left[
2a^2 (1-\pi i k)\exp\left( \frac{\pi i kt}{a} \right) 
-2a^2 - 2\pi i a k(t+a) + \pi^2 k^2 t(t+2a)
 \right]}, & \\[10pt] 
\qquad \qquad -2a \leq t \leq 0.
\end{cases}
\]
Applying the Weil explicit formula \eqref{Eq_515} 
to $\phi_{1,k}(t)$ and $\phi_{2,k}(t)$, 
we obtain 
\begin{equation} \label{Eq_517}
\aligned 
\sum_\gamma & \frac{(\cos(a\gamma)-1)}{\gamma^2} \cdot 
 \frac{(-1)^k2a\sin(a\gamma)}{k\pi+a\gamma} \\
& = (-1)^{k}32 a^2 i (\cos(ia/2)-1)
 \frac{\sin(ia/2)}{4\pi^2 k^2 +a^2} \\
& \quad - \sum_{1 \leq n \leq e^a} \frac{\Lambda(n)}{\sqrt{n}} 
\frac{a^2}{\pi^2k^2}\left[
 ((-1)^k-2)\cos\left( \frac{\pi k \log n}{a} \right) + (-1)^k 
\right] \\
& \quad -  
\sum_{e^a < n \leq e^{2a}} \frac{\Lambda(n)}{\sqrt{n}} 
\frac{a^2(-1)^k}{\pi^2k^2}\left[
 \cos\left( \frac{\pi  k \log n}{a} \right) 
-1 \right] \\
& \quad - \frac{a^2}{\pi^2k^2}((-1)^k-1)\log \pi \\
& \quad + 8a^2 \sum_{n=0}^{\infty}
\frac{ (e^{-a(4n+1)}-2e^{-a(4n+1)/2})(-1)^{k+1}}{(4n+1)(a^2(4n+1)^2+4\pi^2k^2)} \\
& \quad 
+ \frac{a^2(-1)^k (1-2(-1)^k)}{4\pi^2k^2} 
\left[ 
\frac{\Gamma'}{\Gamma}\left(\frac{1}{4}+\frac{\pi i k}{2a}\right) 
+
\frac{\Gamma'}{\Gamma}\left(\frac{1}{4}-\frac{\pi i k}{2a}\right) 
\right]
\\
& \quad 
+ \frac{a^2(-1)^k }{2\pi^2 k^2}
\frac{\Gamma'}{\Gamma}\left(\frac{1}{4}\right) 
\endaligned 
\end{equation}
and 
\begin{equation} \label{Eq_518}
\aligned
\sum_\gamma & \frac{a\gamma\cos(a\gamma)-\sin(a\gamma)}{a \gamma^3} \cdot 
 \frac{(-1)^k2a\sin(a\gamma)}{k\pi+a\gamma} \\
& = 64 a \Bigl( (ai/2)\cos(ai/2)-\sin(ai/2) \Bigr)
\frac{(-1)^{k}\sin(ia/2)}{4\pi^2 k^2 +a^2} \\
& \quad - a^2\sum_{1 \leq n \leq e^{2a}} \frac{\Lambda(n)}{\sqrt{n}} 
\frac{(-1)^k}{\pi^3k^3}\sin\left( \frac{\pi  k \log n}{a} \right)  \\
& \quad +  a\sum_{1 \leq n \leq e^{2a}} \frac{\Lambda(n)}{\sqrt{n}} 
\frac{(-1)^k}{\pi^2k^2} 
\left( (\log n-a )-a  \cos\left( \frac{\pi  k \log n}{a} \right) \right) \\
& \quad - \frac{a^2(-1)^k}{\pi^2k^2} \log \pi \\
& \quad + 8a
\sum_{n=0}^{\infty}\frac{(a(4n+1)+2)e^{-a(4n+1)}(-1)^{k+1}}{(4n+1)^2(a^2(4n+1)^2+4\pi^2 k^2)} \\
& \quad + \frac{a^2i(-1)^k }{4\pi^3k^3} 
\left[
\frac{\Gamma'}{\Gamma}\left(\frac{1}{4}+\frac{\pi ik}{2a}\right) 
(1-\pi ik)
 - 
\frac{\Gamma'}{\Gamma}\left(\frac{1}{4}-\frac{\pi ik}{2a}\right) 
(1+\pi ik)
\right] \\
& \quad + \frac{a(-1)^k}{4\pi^2k^2} \left[  
2a \frac{\Gamma'}{\Gamma}\left(\frac{1}{4}\right) 
+ \psi^{(1)}\left(\frac{1}{4}\right) 
\right]
,
\endaligned 
\end{equation}
respectively, where $\psi^{(1)}(z) := \frac{d^2}{dz^2} \log \Gamma(z)$. 
The calculation for the integral on the right-hand side of \eqref{Eq_515} 
is performed in the same way as in \cite[\S5]{Yo92}.
By substituting \eqref{Eq_517} and \eqref{Eq_518} into \eqref{Eq_516}, 
\begin{equation} \label{Eq_519} 
\langle \chi_0, \chi_k \rangle \ll \frac{\log k}{k^2} , 
\end{equation}
where the implied constant depends only on $a$. 
On the other hand, we have 
\[
\aligned 
\langle \chi_0, D(\tau_j^\pm) \rangle
&= \pm \frac{(-1)^{j-1}\pi i}{a\sqrt{j(j+1)}} 
\left[ \sum_{k=1}^{j} (-1)^{k-1}(N+k)\langle \chi_0, \chi_{\pm(N+k)} \rangle 
\right. \\
& \qquad \qquad \qquad \qquad \qquad \left.
-j(-1)^{(j+1)-1}(N+j+1)\langle \chi_0, \chi_{\pm(N+j+1)} \rangle 
\right]
\endaligned 
\]
by \eqref{Eq_505}. 
Therefore, 
\[
\langle \chi_0, D(\tau_j^\pm) \rangle \ll \frac{(\log j)^2}{j} 
\]
by  \eqref{Eq_519}, 
where the implied constant depends only on $a$. 
This implies \eqref{Eq_511} by definition of $\eta_j$. 

\subsection{Comparison with Yoshida's results} 

For every $0<a<\infty$, there exists $N_0=N_0(a)>0$ such that 
$\langle\cdot,\cdot\rangle_W$ is positive definite 
on $K_N(a)$ if $N \geq N_0$ (\cite[Lemma 3]{Yo92}). 
Therefore, we can take the completion $\widehat{K_N(a)}$ 
of $K_N(a)$ with respect to $\langle\cdot,\cdot\rangle_W$. 
Define $W(a)$ by $K(a) = W(a) \oplus K_N(a)$ and set 
$\widehat{K(a)} := W(a) \oplus \widehat{K_N(a)}$. 
Then the hermitian form can be extended to $\widehat{K(a)}$. 
Yoshida proved that the RH is true if and only if 
$\langle\cdot,\cdot\rangle_W$ is non-degenerate 
on $\widehat{K(a)}$ for every $a>0$. 
Further, he proved that $\langle\cdot,\cdot\rangle_W$ 
is non-degenerate on $C(a)$ and $K(a)$ for every $a>0$ (\cite[Propositions 2 and 7]{Yo92}). 

Theorem \ref{Thm_1_4} is an analog of Yoshida's result, 
but it has the advantage that the space is simpler than $\widehat{K(a)}$ 
and the hermitian forms are described by an integral operator with continuous kernel. 
Further, Proposition \ref{Prop_3_1} and 
Yoshida's result \cite[Proposition 2]{Yo92} imply that 
$\langle \cdot, \cdot \rangle_{G_g,a}$ is non-degenerate on $\mathfrak{C}_0(a)$. 

%
\section{Properties of the kernel $G_g(t,u)$} \label{Section_6}
%

In this part we prove Theorem \ref{Thm_1_5} 
and study a little about eigenvalues of $\mathsf{G}_g[a]$. 

\subsection{Proof of Theorem \ref{Thm_1_5}} 

Using $g_\ast(t)$ in \eqref{Eq_408}, we define 
$G_\ast(t,u)=g_\ast(t-u)-g_\ast(t)-g_\ast(-u)+g_\ast(0)$ 
for $\ast \in \{0,1,\infty\}$ so that 
$G_g(t,u)=G_0(t,u)+G_1(t,u)+G_\infty(t,u)$. 
It is known that a continuous integral kernel $K(t,u)$ on $[-a,a]\times[-a,a]$ 
is of trace class if it satisfies the Lipschitz condition 
$|K(t,u_1)-K(t,u_2)| \leq C |u_1-u_2|^\alpha$ 
for some $C>0$ and $1/2 < \alpha \le 1$ (\cite[Chapter IV, Theorem 8.2]{GGK01}). 
Therefore, $G_0(t,u)$ and $G_1(t,u)$ are trace class kernels,  
but the Lipschitz condition can not be applied to $G_\infty(t,u)$, 
because 
\[
C - e^{-t/2}\Phi(e^{-2t},2,1/4) = 2 t\log(1/t) + A \,t 
- \sum_{k=2}^{\infty} \zeta(2-k,1/4) \frac{(-2t)^k}{k!}
\]
for $t>0$ by using the series expansion of $\Phi(z,s,a)$ for $z$ 
in \cite[p.30, (9)]{EMOT81}, where $A = -2(\log 2 + \psi(1/4)-\psi(2)) = \pi+4\log 2+2$ 
and $\zeta (2-k,1/4)=-B_{k-1}(1/4)/(k-1)$ with Bernoulli polynomials $B_n(x)$. 
However, if we decompose the kernel as 
$G_{\infty}(t,u)=G_{\infty,1}(t,u)+G_{\infty,2}(t,u)$ 
according to $g_\infty(t)=2|t|\log(1/|t|)+(g_\infty(t)-2|t|\log(1/|t|))=g_{\infty,1}(t)+g_{\infty,2}(t)$, 
then $G_{\infty,2}(t,u)$ is a trace class kernel.  
Therefore, it remains to show that $G_{\infty,1}(t,u)$ is of trace class. 

If a Hilbert--Schmidt kernel  $K(t,u)$ on $[-a,a]\times[-a,a]$ 
has the derivative in the mean  $(\partial/\partial u)K(t,u)$ 
which is also a Hilbert--Schmidt kernel, 
then $K(t,u)$ defines a trace class operator 
(\cite[p. 120, 3]{GoKr69}). 
Since $g_{\infty,1}(t)=2|t|\log(1/|t|)$ is a piecewise continuously differentiable function 
and its derivative is absolutely integrable on $[-a,a]$, 
the kernel $G_{\infty,1}(t,u)$ is of trace class.   \hfill $\Box$
\bigskip

If $\widehat{\phi}(0)=0$ for $\phi \in L^2(-a,a)$, we have 
\[
\langle \phi, \phi \rangle_{G_g,a} 
= \int_{-a}^{a}\left( \int_{-a}^{a} g(t-u) \phi(u) \, du \right) \overline{\phi(t)}\, dt. 
\]
From this, if $G_g(t,u)$ is non-negative definite, 
one may expect $g(t-u)$ to be so. 
However it is not the case, because $g(t-u)$ is a trace class kernel without 
any assumptions by the proof of Theorem \ref{Thm_1_5}, 
and $\int g(t-t) dt=\int g(0) dt =0$. 
Therefore, if $g(t-u)$ is assumed to be non-negative definite, 
we have $g(t)=0$. 
This is a contradiction. 

\subsection{Eigenvalues} 

For $0<a<\infty$, we study the eigenvalues of the integral operator $\mathsf{G}_g[a]$ 
defined in \eqref{eq_110}. 
Let $\phi(t)$ be the eigenfunction of $\mathsf{G}_g[a]$ 
with the eigenvalue $\lambda$. Then 
\[
\aligned 
\lambda \phi(t) 
& = \mathbf{1}_{[-a,a]}(t) \sum_{\gamma} 
\left( \frac{\widehat{\phi}(-\gamma)-\widehat{\phi}(0)}{\gamma} \right)
\frac{e^{i\gamma t}-1}{\gamma} 
\endaligned 
\]
by formula \eqref{Eq_109} of the kernel $G_g(t,u)$. 
If $\lambda\not=0$, we can write 
\[
\phi(t) = \mathbf{1}_{[-a,a]}(t)
\sum_{\gamma} w_\gamma \frac{e^{i\gamma t}-1}{\gamma} \quad \text{with} \quad 
w_\gamma = \frac{1}{\lambda}\frac{\widehat{\phi}(-\gamma)-\widehat{\phi}(0)}{\gamma} 
\in \C.
\] 
In this case, we have $\phi(0)=0$ and 
\[
\aligned 
\frac{\hat{\phi}(-\gamma)-\hat{\phi}(0)}{\gamma}
& =\sum_{\mu} w_\mu 
\int_{-a}^{a}
\frac{e^{i\mu t}-1}{\mu} \frac{e^{-i\gamma t}-1}{\gamma} \, dt\\
& = \frac{2a}{\gamma}
\sum_{\mu} \frac{w_\mu}{\mu}
\left(
\frac{\sin(a(\gamma-\mu))}{a(\gamma-\mu)} 
- \frac{\sin(a \gamma)}{a\gamma} 
- \frac{\sin(a \mu)}{a\mu} +1
\right).
\endaligned 
\]
Therefore, 
\[
\lambda w_\gamma
=
\frac{2a}{\gamma}
\sum_{\mu} \frac{w_\mu}{\mu}
\left(
\frac{\sin(a(\gamma-\mu))}{a(\gamma-\mu)} 
- \frac{\sin(a \gamma)}{a\gamma} 
- \frac{\sin(a \mu)}{a\mu} + 1
\right).
\]
Hence, if we put 
\[
H(x,y;a)
= \frac{2a}{xy}
\left(
\frac{\sin(a(x-y))}{a(x-y)} 
- \frac{\sin(ax)}{ax} 
- \frac{\sin(ay)}{ay} 
+1 \right), 
\]
then the non-zero eigenvalues of $\mathsf{G}_g[a]$ correspond to 
the eigenvalues of the linear system
\[
\lambda w_\gamma
=
\sum_{\mu}  H(\gamma,\mu;a) \, w_{\mu}
\]
for $w_\gamma$'s, 
where $\gamma$, $\mu$ runs over the zeros of $\xi(1/2-iz)$ 
counting with multiplicity. 
The system is an analog of the linear system introduced 
in Bombieri \cite[Section 7]{Bo01} to study Weil's hermitian form $\langle \cdot, \cdot \rangle_W$.  

On the other hand, if $\lambda=0$, we have 
\[
\sum_{\gamma} w_\gamma \frac{e^{i\gamma t}-1}{\gamma^2} =0 
\] 
for $t \in (-a,a)$ by writing $w_\gamma=\widehat{\phi}(-\gamma)-\widehat{\phi}(0)$. 
Further, $w_\gamma\not=0$ for some $\gamma$ 
by Lemma \ref{Lem_2_1}, 
since $\widehat{\phi}(z)$ is a non-constant entire function of the exponential type. 
By Theorem \ref{Thm_1_4}, $\lambda=0$ is actually an eigenvalue 
when the RH is false. 
Therefore, we obtain the following result similar to 
(iii) in the introduction or the corollary of Theorem 11 in \cite{Bo01}. 

\begin{theorem} 
Suppose that the RH is false. Then there exists $a_0>0$ 
and a non-identically vanishing sequence of complex numbers $\{w_\gamma\}_{\gamma}$ such that 
\[
\sum_{\gamma} w_\gamma \frac{e^{i\gamma t}-1}{\gamma^2} = 0
\]
on the interval $(-a_0,a_0)$ as a function of $t$. 
\end{theorem}

%
\section{Proofs of Theorems \ref{Thm_1_6}, \ref{Thm_1_7}, and \ref{Thm_1_8} } \label{Section_7}
%

\subsection{Proof of Theorems \ref{Thm_1_6}} 

Suppose that the RH is true. 
Then all $\gamma$ in \eqref{Eq_103} are real. 
Therefore, noting that $\xi(1/2-iz)$ is an entire function of order one, 
$|\Psi(t)| \leq 2\sum_{\gamma} |\gamma|^{-2} < \infty$. 

Conversely, suppose that $\Psi(t)$ is bounded on $[0,\infty)$. 
Then, the integral on the left-hand side of \eqref{Eq_102} 
converges absolutely and uniformly on any compact subset in $\C_+$. 
Hence, $(\xi'/\xi)(1/2-iz)$ has no poles in  $\C_+$, 
which implies that the RH is true. \hfill $\Box$
\medskip

From Theorem \ref{Thm_1_6} we are interested in the supreme of values of $\Psi(t)$. 
It is clearly less than or equal to $2\sum_\gamma \gamma^{-2}\,(<0.094)$ 
by \eqref{Eq_103} under the RH. 
However, it is not easy to determine the the limit superior 
and limit inferior of $\Psi(t)$ even if assuming the RH. 
Assuming that all zeros $\gamma$ of $\xi(1/2-iz)$ are real (RH) and simple 
and that the set of all positive zeros is linearly independent over the rationals, 
we can prove 
\[
\limsup_{t \to \infty} \Psi(t)=2\sum_\gamma \frac{1}{\gamma^2}
\quad \text{and} \quad 
\liminf_{t \to \infty} \Psi(t)=0
\]
by Kronecker's theorem in Diophantine approximations. 
The conjectural value of the limit inferior of $\Psi(t)$ suggests 
the difficulty of proving the non-negativity of $\Psi(t)$ by approximation.

\subsection{Proof of Theorems \ref{Thm_1_7}} 

Suppose that the RH is true. 
Then all $\gamma$ in \eqref{Eq_103} are real. 
Therefore, each term of the middle sum is non-negative. 
Also, for each $t>0$, not all terms are simultaneously zero 
by Lemma \ref{Lem_2_1}. 
Hence $\Psi(t)>0$ for $t>0$. 

Conversely, suppose that $\Psi(t) \geq 0$ for $t>0$. 
Because $\zeta(s)$ has no zeros in the half-line $[1/2,\infty)$ 
by the series expansion 
$(1-2^{1-s})\zeta(s) = \sum_{n=1}^{\infty} (-1)^{n-1}n^{-s}$ for $\Re(s)>0$, 
the logarithmic derivative $(\xi'/\xi)(1/2-iz)$ 
has no poles in the half-line $i\cdot[0,\infty)$ of the imaginary axis. 
Then, by the integral formula \eqref{Eq_102} 
and the well-known result for the Laplace transform for non-negative functions 
(\cite[Theorem 5b in Chap. II]{Wi41}),
 $(\xi'/\xi)(1/2-iz)$ has no poles in  $\C_+$, 
which implies that the RH is true. \hfill $\Box$

\subsection{Proof of Theorem \ref{Thm_1_8}} 

Suppose that the RH is true. Then $\Psi(t)>0$ for $t>0$ by Theorem \ref{Thm_1_7}. 
Therefore, $\{\mu_n\}_{n\geq 0}$ is a Stieltjes moment sequence 
for the measure $4^{-1}e^{-t/2}\Psi(t)\,dt$ on $[0,\infty)$. 
Hence $\det \Delta_n \geq 0$ and $\det \Delta_n^{(1)} \geq 0$ for all $n \in \Z_{>0}$ 
(\cite[Chapter V, \S1]{KrNu77}). 
Conversely, suppose that 
$\det \Delta_n \geq 0$ and $\det \Delta_n^{(1)} \geq 0$ for all $n \in \Z_{>0}$. 
Then $\{\mu_n\}_{n\geq 0}$ is a Stieltjes moment sequence of some measure 
on $[0,\infty)$ (\cite[Chapter V, \S1]{KrNu77}). 
On the other hand, 
the estimate $4^{-1}e^{-t/2}\Psi(t) \ll e^{-c\sqrt{t}}$ obtained from 
Proposition \ref{Prop_2_1} implies that the Stieltjes moment problem 
for $\{\mu_n\}_{n\geq 0}$ is uniquely determined by \cite[Theorem 2]{Lin17}. 
Hence $\{\mu_n\}_{n\geq 0}$ must be equal to the moment sequence 
of the measure $4^{-1}e^{-t/2}\Psi(t)\,dt$, 
and thus $\Psi(t) \geq 0$ for $t>0$. \hfill $\Box$ 

%
\section{Relations with Li coefficients} \label{Section_8}
%

Comparing formulas \eqref{Eq_114} and \eqref{Eq_115}, 
the similarity between the moments $\{\mu_n\}$ 
and the Li coefficients $\{\lambda_{n+1}\}$ is obvious. 
In fact, they have the following explicit relation. 

\begin{theorem} 
Li coefficients \eqref{Eq_114} are represented by the moments \eqref{Eq_113}  as 
\begin{equation} \label{Eq_801}
\lambda_1=\mu_0, \quad 
\lambda_2 = 6\mu_0-\mu_1, \quad 
\lambda_3 = 19\mu_0-7\mu_1 + \frac{1}{2}\mu_2
\end{equation}
and 
\begin{equation} \label{Eq_802}
\frac{1}{n!} \,\lambda_{n} = \sum_{k=0}^{n-1} 
\frac{1}{k!(k+3)!(n-k-1)!}
\left(\Bigl(k - \frac{4n-1}{2} \Bigr)^2 + 2n + \frac{7}{4}\right)
(-1)^k \mu_k
\end{equation}
for $n \geq 4$. 
Conversely, the moments are represented by the Li coefficients as 
\begin{equation} \label{Eq_803}
\frac{1}{n!} \, \mu_n 
= 
\sum_{j=1}^{n+1} 
\left[
\sum_{k=1}^{n-j+2} k\, 2^{k-1} 
\binom{n-k+2}{j} 
\right] (-1)^{j+1}\lambda_j 
\end{equation}
for all non-negative integers $n$. 
The relations \eqref{Eq_801}--\eqref{Eq_802} and \eqref{Eq_803} are inverse to each other.
\end{theorem}
\begin{proof}
We first prove \eqref{Eq_801} and \eqref{Eq_802}. 
Formula \eqref{Eq_114} means that Li coefficients are 
defined by the power series expansion 
\begin{equation} \label{Eq_804}
\sum_{n=0}^{\infty} \lambda_{n+1} w^n 
= 
\frac{1}{(1-w)^2} \, \frac{\xi'}{\xi}\left(\frac{1}{1-w}\right).
\end{equation}
By the changing of variables $z=(i/2)(1+w)/(1-w)$, we have 
\begin{equation} \label{Eq_805}
\aligned 
\frac{1}{(1-w)^2} \frac{\xi'}{\xi}\left(\frac{1}{1-w}\right)
& = -\frac{z^2}{(1-w)^2} \left[ -\frac{1}{z^2}  \, \frac{\xi'}{\xi}\left(\frac{1}{2}-iz\right) \right] \\
& = \frac{1}{4}\frac{(1+w)^2}{(1-w)^4}
\int_{0}^{\infty} \Psi(t) \exp\left( -\frac{t}{2}\cdot \frac{1+w}{1-w} \right) \,dt \\
& = \sum_{n=0}^{\infty} 
\int_{0}^{\infty} 4^{-1}e^{-t/2} \Psi(t)  P_n(t)\, dt \cdot w^n, 
\endaligned 
\end{equation}
where $P_n(t)$ are some polynomials. 
The exchanging of the order of the integral and sum is justified 
by the estimate in Proposition \ref{Prop_2_1}. 
Comparing \eqref{Eq_804} and \eqref{Eq_805}, 
we have $\lambda_{n+1}=\sum_k c_{nk} \mu_k$ 
if $P_n(t) = \sum_k c_{nk} t^k$. 
Therefore, we then calculate $P_n(t)$ explicitly to conclude 
\eqref{Eq_801} and \eqref{Eq_802}.
Using the power series of the exponential, 
\[
\aligned 
\exp & \left( -\frac{t}{2}\cdot \frac{1+w}{1-w} \right) 
= e^{-t/2} \exp\left(-\frac{wt}{1-w} \right) \\
&= e^{-t/2}
 + e^{-t/2} (-t)  \sum_{l=0}^{\infty} w^{l+1}
 + e^{-t/2} \sum_{k=2}^{\infty} \frac{(-t)^k}{k!} w^k 
\frac{1}{(k-1)!} \sum_{l=0}^{\infty} \prod_{i=1}^{k-1}(l+i) w^l. 
\endaligned 
\]
On the other hand, 
\[
\frac{(1+w)^2}{(1-w)^4}
= \frac{1}{3}\sum_{m=0}^{\infty} (m+1)(2m^2+4m+3)w^m. 
\]
Multiplying these two expansions, 
\[
\aligned 
\sum_{n=0}^{\infty} & P_n(t) \,w^n
= 
e^{t/2} \cdot \frac{(1+w)^2}{(1-w)^4} 
\exp\left( -\frac{t}{2}\cdot \frac{1+w}{1-w} \right) 
\\
& = 1 + (6-t)w + \frac{1}{3}\sum_{m=2}^{\infty} (m+1)(2m^2+4m+3)w^m \\
& \quad + (-7t)\sum_{l=0}^{\infty} w^{l+2} 
+ (-t)\frac{1}{3}\sum_{m=2}^{\infty} (m+1)(2m^2+4m+3)\sum_{l=0}^{\infty} w^{m+l+1} 
\\
& \quad 
+ \sum_{k=2}^{\infty} \sum_{l=0}^{\infty}\sum_{m=0}^{\infty}
(-t)^k\frac{\prod_{i=1}^{k-1}(l+i)}{3k!(k-1)!} 
(m+1)(2m^2+4m+3) \cdot w^{k+l+m} .
\endaligned 
\]
Rearranging the right-hand side concerning monomials $w^k$, we obtain
\[
\aligned 
\sum_{n=0}^{\infty} & P_n(t) \,w^n
= 1 + (6-t)w + \left(19-7t + \frac{t^2}{2}\right) w^2  \\ 
& \quad + \sum_{n=3}^{\infty}\left[ 
\frac{(n+1)(2(n+1)^2+1)}{3} -t
\left( 7 + 
\sum_{m=2}^{n-1}\frac{(m+1)(2(m+1)^2+1)}{3}
\right) 
\right. \\
& \quad \left.
\qquad \qquad + 
\sum_{k=2}^{n} \frac{(-t)^k}{k!(k-1)!}
\sum_{m=0}^{n-k} \frac{(m+1)(2(m+1)^2+1)}{3}
\prod_{i=1}^{k-1}(n-k-m+i)
 \right] w^{n}. 
\endaligned 
\]
On the right-hand side, we have 
\[
7 + \sum_{m=2}^{n-1}\frac{(m+1)(2(m+1)^2+1)}{3} 
= \frac{n(n+1)(n^2+n+1)}{6}, 
\]
\[
\prod_{i=1}^{k-1}(n-k-m+i)  = \frac{(n-m-1)!}{(n-k-m)!}, 
\]
and
\[
\aligned 
\sum_{m=0}^{n-k} & \frac{(m+1)(2(m+1)^2+1)}{3} \frac{(n-m-1)!}{(n-k-m)!}\\
& = \frac{1}{k(k+1)(k+2)(k+3)} \, \frac{(n+1)!}{(n-k)!}
\left(\Bigl(k - \frac{4(n+1)-1}{2} \Bigr)^2 + 2(n+1) + \frac{7}{4} \right) 
\endaligned 
\]
by elementary calculations. Therefore, 
\[
\aligned 
\sum_{n=0}^{\infty} & P_n(t) \,w^n 
= 1 + (6-t)w + \left(19-7t + \frac{t^2}{2}\right) w^2  \\ 
& \quad + \sum_{n=3}^{\infty}\left[ 
\frac{(n+1)(2(n+1)^2+1)}{3} -t
 \frac{n(n+1)(n^2+n+1)}{6}
\right. \\
& \quad \left.
\qquad \qquad + 
\sum_{k=2}^{n} \frac{(-t)^k}{k!(k+3)!}\frac{(n+1)!}{(n-k)!}
\left(\Bigl(k - \frac{4(n+1)-1}{2} \Bigr)^2 + 2(n+1) + \frac{7}{4} \right)
 \right] w^{n}. 
\endaligned 
\]
In the sum for $n \geq 3$, we find that the first and second terms 
are equal to cases of $k=0$ and $k=1$, respectively, 
for the formula of coefficients of  $t^k$ for $k \geq 2$. 
Hence, we obtain the relations \eqref{Eq_801} and \eqref{Eq_802}. 

We second prove \eqref{Eq_803}. 
By the changing of variables $w=X/(1-X)$ in \eqref{Eq_804}, we have 
\begin{equation} \label{Eq_806}
\frac{1}{(1-2X)^2(1-X)^2}
\sum_{n=0}^{\infty} \lambda_{n+1} \left( \frac{X}{X-1}\right)^n
= 
\sum_{n=0}^{\infty} 
\mu_n \frac{X^n}{n!}.
\end{equation}
Applying $\lambda_n=\sum_\rho (1-(1-\rho^{-1})^n)$ in \cite[(1.4)]{Li97} 
to the left-hand side, 
\begin{equation} \label{Eq_807}
\frac{1}{(1-2X)^2(1-X)^2}
\sum_{n=0}^{\infty} \lambda_{n+1} \left( \frac{X}{X-1}\right)^n
= 
\sum_{n=0}^{\infty} \left[ 
\sum_{k=1}^{n+1} k\, 2^{k-1} \sum_\rho \rho^{-(n-k+2)}
\right] X^n. 
\end{equation}
On the other hand, $\lambda_n=\sum_\rho (1-(1-\rho^{-1})^n)$ gives
\begin{equation} \label{Eq_808}
\sum_\rho \rho^{-m} = \sum_{k=1}^{m} (-1)^{k+1}\binom{m}{k} \lambda_k 
\end{equation}
in an inductive procedure. 
Substituting \eqref{Eq_808} into \eqref{Eq_807} 
and comparing its coefficients with \eqref{Eq_806}, we obtain the relation \eqref{Eq_803}. 

Since both relations \eqref{Eq_801}--\eqref{Eq_802} and \eqref{Eq_803} 
are obtained by different power series expansions of \eqref{Eq_804}, 
it is clear that they are inverse to each other.
\end{proof}

If the RH is true, both $\{\lambda_n\}$ and $\{\mu_n\}$ are positive, 
but unfortunately in \eqref{Eq_801}--\eqref{Eq_802} and \eqref{Eq_803}, 
the positivity of one does not directly lead to the positivity of the other. 
However, as an application of the relation \eqref{Eq_801}--\eqref{Eq_802}, 
the following recurrence formula of the moments is obtained.

\begin{theorem} 

Let $\{a_j\}_{j \geq 1}$ be coefficients 
of the power series expansion 
$\xi(1/(1-w))=1+\sum_{j=1}^{\infty} a_j w^j$.  
For $n \in \Z_{>0}$ and $k \in \Z_{\geq 0}$ with $n \geq k$, 
we set
\begin{equation*} 
\aligned 
b_{n,k} &:= \frac{n!}{k!(k+3)!} 
\left[
\frac{(n+1)!}{(n-k)!}
\Bigl(
\Bigl(k - \frac{4n+3}{2} \Bigr)^2 + 2n + \frac{15}{4}\Bigr)
 \right. \\
&
\qquad \qquad \qquad \qquad \qquad \left. 
+
\sum_{j=k+1}^{n}  
\frac{j!}{(j-k-1)!}
\Bigl((2j-k)^2+k+2\Bigr) \, a_{n-j+1} 
 \right]. 
\endaligned 
\end{equation*}
Then, all $b_{n,k}$ are positive and the recurrence relation 
\begin{equation} \label{Eq_809}
(-1)^{n}\mu_n =  (n+1)!\,a_{n+1} -\sum_{k=0}^{n-1} (-1)^{k}b_{n,k} \,\mu_k 
\end{equation}
holds for all non-negative $n$. 
\end{theorem}
\begin{proof} 
The positivity of $b_{n,k}$ follows from the positivity of $a_j$ 
in \cite[p. 327]{Li97}. 
We recall that the recurrence formula
\begin{equation} \label{Eq_810}
\lambda_{n+1} = (n+1)a_{n+1} - \sum_{j=1}^{n}  a_{n-j+1} \lambda_j
\end{equation}
holds for every non-negative integer $n$ in \cite[p. 327]{Li97}. 
Substituting \eqref{Eq_801} to \eqref{Eq_810}  
and then calculating a little, 
\eqref{Eq_809} is obtained for $n=1$ and $2$. 
For $n \geq 3$, 
substituting \eqref{Eq_801} and \eqref{Eq_802} for both sides of \eqref{Eq_810}, 
then moving the terms other than $(-1)^n\mu_n/n!$ on the left-hand side 
to the right-hand side, 
and finally grouping the right-hand side concerning moments $\mu_k$, we obtain \eqref{Eq_809}. 
\end{proof}

%
\section{Relation with strings under the RH} \label{Section_9}
%

For a Kre\u{\i}n string $S[m, L]$ which consists of 
$0<L \leq \infty$ 
and a 
right-continuous 
nondecreasing 
non-negative function $m(x)$ on $[0,L)$ with $m(0_-)=0$, 
we take solutions $\phi(x,z)$ and $\psi(x,z)$ 
of the string equation 
$dy'(x) + z y(x) dm(x)=0$ on $[0,L)$ 
satisfying the initial condition $\phi'(0,z)=\psi(0,z)=0$, $\phi(0,z)=\psi'(0,z)=1$. 
Then the Titchmarsh--Weyl function $q(z) = \lim_{x \to L} \psi(x,z)/\phi(x,z)$ 
exists and belongs to the subclass $\mathcal{N}_S$ 
of the Nevanlinna class $\mathcal{N}$ consisting of $q(z)$ such that 
\[
q(z) = b + \int_{0}^{\infty} \frac{d\sigma(\lambda)}{\lambda-z} 
\]
for some $b \geq 0$ and a measure $\sigma$ on $[0,\infty)$ 
with $\int_{0}^{\infty} d\sigma(\lambda)/(\lambda+1) < \infty$. 
Kre\u{\i}n \cite{Kr52} proved that the correspondence 
$S[m,L] \mapsto \mathcal{N}_S$ is bijective. 

Using the meromorphic function $Q_{\xi}(z)$ in \eqref{Eq_107}, we define 
\[
q_\xi(z) = \frac{Q_\xi(\sqrt{z})}{\sqrt{z}}
\]
and suppose that the RH is true.  
Then, both $Q_\xi(z)$ and $q_\xi(z)$ belong to $\mathcal{N}$. 
Moreover, we have 
\[
\aligned 
q_\xi(z) 
& 
=  \frac{i}{\sqrt{z}} \frac{\xi'}{\xi}\left(\frac{1}{2}-i\sqrt{z}\right) 
= \sum_{\gamma>0} \frac{2}{\gamma^2-z} 
= b + \int_{0}^{\infty} \frac{d\sigma(\lambda)}{\lambda-z} \quad (b=0)
\endaligned 
\]
where $\sigma$ is a measure on $[0,\infty)$ supported on points $\gamma^2$. 
In other words, $q_\xi \in \mathcal{N}_S$. 
Hence, there exists a string $S[m_\xi,L_\xi]$ corresponding to $q_\xi(z)$. 
Such string is called Zeta string in Kotani~\cite[\S3.2]{Ko21}, 
where he proved that $m_\xi(x) \sim 4x(\log x)^{-2}$ as $x \to 0+$. 

%
\section{Mean Periodicity} \label{Section_10}
%

Here we explain that the screw function $g(t)$ 
in \eqref{Eq_101} and \eqref{Eq_108} is mean-periodic 
even without assuming the RH. 
Let 
\[
\eta(t):=  
2 e^{-t/2}\sum_{n=1}^{\infty} \bigl( 2\pi^2 n^4 e^{-4t} - 3 \pi n^2 e^{-2t} \bigr)
\exp(-\pi n^2 e^{-2t}). 
\]
Then $\xi(1/2-iz) = \widehat{\eta}(z)$ for all $z \in \C$ (\cite[\S10.1]{Tit86}). 
In addition, $-iz \, \xi(1/2-iz) = \widehat{\eta'}(z)$ for all $z \in \C$, 
since $\eta(t)$ is smooth and decays faster than any exponential as $|t| \to +\infty$. 
Therefore, by using \eqref{Eq_103} and putting $\phi=\eta'$, we have 
\[
(g \ast \phi)(t) 
= \sum_{\gamma} \frac{1}{\gamma^2} (e^{i\gamma t} \widehat{\phi}(-\gamma) - \widehat{\phi}(0) ) =0.
\]
Hence, if we chose a suitable function space $X$ consisting of functions on $\R$ 
such that $g \in X$ and $\phi=\eta'$ belongs to the dual space $X^\ast$, 
then $g$ is a $X$-mean periodic function 
(see \cite[\S2]{FRS12} for a quick overview of mean-periodic functions).  
In fact, we can chose the spaces $\mathcal{C}_{\rm exp}^\infty(\R)$ 
and $\mathbf{S}(\R)^\ast$ in \cite{FRS12} 
as such a space $X$ by Theorem \ref{Thm_1_1} (3). 
For the same reason that the convolution equation $g\ast \phi=0$ holds, 
if we chose the space $X$ such that the dual space $X^\ast$ contains the space 
\[
\mathcal{Z}:=
\left\{\left.
\frac{d}{dt}\left( 
e^{-t/2}\sum_{n=1}^{\infty} f(ne^{-t}) \right) 
~\right|~ f \in S(\R),\, f(x)=f(-x),\,f(0)=\widehat{f}(0)=0
\right\}
\]
like $\mathcal{C}_{\rm exp}^\infty(\R)$ 
and $\mathbf{S}(\R)^\ast$, 
then the space $\mathcal{T}(g)\subset X$ spanned by all translations $g(t-r)$ by $r \in \R$ 
is orthogonal to $\mathcal{Z}$ with respect to the pairing $(\cdot,\cdot): X \times X^\ast \to \C$. 
The space $\mathcal{Z}$ consists of the first derivatives of functions 
in the space introduced by A. Connes  in \cite[Appendix I]{Co99} (see also \cite[\S6.1]{CC21b} 
and R. Meyer \cite{Me05}) 
to study the spectral realization of the zeros of $\xi(s)$. 
In the above sense, the screw function $g$ ``generates'' an orthogonal complement 
(which, depending on $X$, may be a subspace of this by the existence of trivial zeros of $\zeta(s)$) 
of Connes' space. This is a situation similar to \cite[Theorem 3.1]{Su12}. 

Independent of the choice of function space, 
the convolution equation $g\ast \phi=0$ yields the integral representation 
\[
\frac{\widehat{g\ast \phi^+}(z)}{\widehat{\phi^+}(z)} 
= -\frac{\widehat{g\ast \phi^-}(z)}{\widehat{\phi^-}(z)} = \frac{1}{z^2}\, \frac{\xi'}{\xi}\left(\frac{1}{2}-iz \right)
\]
that holds for all $z \in \C$ by the Fourier--Carleman transform (\cite[Definition 2.8]{FRS12}), 
where $\phi^+(t)=\phi(t)$ if $t \geq 0$ and $=0$ otherwise 
and $\phi^-(t)=0$ if $t \geq 0$ and $=\phi(t)$ otherwise.  

Although the screw function $g(t)$ attached to $\xi(s)$ 
is outside the framework of mean-periodic functions studied in \cite{FRS12}, 
we may regard the positivity in Theorem \ref{Thm_1_7} 
an example of the single sign property of mean-periodic functions in \cite[\S4.3]{FRS12}, 
and the equality obtained from \eqref{Eq_101} and \eqref{Eq_103} 
is an analog of the summation formula \cite[Corollary 4.6]{FRS12}. 

%
\section{Variants of $\Psi(t)$}  \label{Section_11}
%

We consider ``shifted'' variants of $\Psi(t)$. 
For a real number $\omega$, we define $\Psi_\omega(t)$ by
\begin{equation} \label{Eq_1101}
\Psi_\omega(t)
:= e^{-\omega t} \Psi(t) 
+ 2\omega \int_{0}^{t} e^{-\omega u} \Psi(u) \, du 
+ \omega^2 \int_{0}^{t} (t-u)e^{-\omega u} \Psi(u) \, du  
\end{equation}
for $t>0$ and extend it as an even function on $\R$. 
Then we have $\Psi_0(t)=\Psi(t)$ and 
\begin{equation} \label{Eq_1102}
\int_{0}^{\infty} \Psi_\omega(t) \,e^{izt} \, dt 
=  -\frac{1}{z^2}\frac{\xi'}{\xi}(\tfrac{1}{2}+\omega-iz), \quad 
\Im(z)>1/2-\omega
\end{equation}
by \eqref{Eq_102} and a little calculation.  
Therefore a nontrivial estimate of $\Psi_\omega(t)$ 
leads to an expansion of the zero-free region of $\xi(s)$ 
as in the case of $\Psi(t)$. 

On the right-hand side of \eqref{Eq_1102}, 
the expansion 
\begin{equation*} 
\Im\left[ i\frac{\xi'}{\xi}(s+\omega) \right] 
= \sum_{\rho} \frac{\Re(s+\omega)-\Re(\rho)}{|s+\omega-\rho|^2} 
\end{equation*}
shows that $i(\xi'/\xi)(1/2+\omega-iz)$ belongs to the class $\mathcal{N}$ 
if  $\xi(s)\not=0$ for $\Re(s)>1/2+\omega$, 
where $\rho$ runs over all zeros of $\xi(s)$ counting with multiplicity.   
Therefore, if $\xi(s)\not=0$ for $\Re(s)>1/2+\omega$, 
the function $\Psi_\omega(t)$ is non-negative 
by \eqref{Eq_1102} and \cite[Satz 5.9]{KrLa77} 
as in the case of $g(t)=-\Psi(t)$. 
Furthermore, 
from formula \eqref{Eq_1102} and the fact that $\xi(s)\not=0$ 
on the positive real line, if we use the same result of Laplace transforms 
as in the proof of Theorem \ref{Thm_1_7}, 
the following seemingly milder equivalent condition is obtained.

\begin{theorem}
Let $\omega \in \R$. Then $\xi(s)\not=0$ for $\Re(s)>1/2+\omega$ 
if and only if there exists $t_0>0$ such that 
$\Psi_\omega(t)$ is non-negative when $t \geq t_0$. 
\end{theorem}

The above result shows that 
$\Psi_\omega(t) \geq 0$ unconditionally if $\omega \geq 1/2$, 
because it is well-known that $\xi(s)\not=0$ for $\Re(s)>1$. 
In contrast, $\Psi_\omega(t)$ changes the sign infinitely many 
if $\omega$ is negative, 
because it is known that $\xi(s)$ has infinitely many zeros 
on the critical line $\Re(s)=1/2$. 
Similar to the relation between $\Psi(t)$ and $\Psi_\omega(t)$, we have 
\begin{equation*} 
\Psi_{\omega+\eta}(t)
= e^{-\eta t} \Psi_\omega(t) 
+ 2\eta \int_{0}^{t} e^{-\eta u} \Psi_\omega(u) \, du 
+ \eta^2 \int_{0}^{t} (t-u)e^{-\eta u} \Psi_\omega(u) \, du
\end{equation*}
for two reals $\omega$ and $\eta$. 
Therefore, the nonnegativity of $\Psi_\omega(t)$ on some interval $[0,t_0)$ 
implies the nonnegativity of 
$\Psi_{\omega+\eta}(t)$ on the same range when $\eta>0$. 
\medskip

For $t>0$, the analogs of \eqref{Eq_101} and \eqref{Eq_103} are as follows:
\[
\aligned 
\Psi_\omega(t) 
& = \sum_{\gamma}\frac{1}{(\gamma^2+\omega^2)^2} 
\Bigl[
\omega t (\gamma^2+\omega^2)
+
(\gamma^2-\omega^2) \\
& \qquad \qquad \qquad 
-
e^{-\omega t} \cdot (\gamma^2 - \omega^2) \cos(\gamma t) 
-
e^{-\omega t} \cdot 2\gamma \omega \sin(\gamma t)
\Bigr] \\
& = 4
\left[
\frac{e^{(\frac{1}{2}-\omega)t}}{(1-2\omega)^2}
+
\frac{e^{-(\frac{1}{2}+\omega)t}}{(1+2\omega)^2} 
-
\frac{4-2(1-\omega t)(1-4\omega^2)}{(1-4\omega^2)^2} 
\right] \\ 
& \quad + 
\frac{t}{2}\left[ \psi
\left(\frac{1}{4}+\frac{\omega}{2}\right) - \log \pi \right] \\
&\quad  + 
\frac{1}{4}\left[ \psi^{(1)}\left(\frac{1}{4}+\frac{\omega}{2}\right) 
- e^{-(\frac{1}{2}+\omega)t}\Phi\left(e^{-2t},2,\frac{1}{4}+\frac{\omega}{2}\right) \right] \\
&\quad  - \sum_{n \leq e^t} \frac{\Lambda(n)}{n^{\frac{1}{2}+\omega}}(t-\log n), 
\endaligned 
\]
where $\psi^{(k)}(s) = \frac{d^{k+1}}{ds^{k+1}}\log \Gamma(s)$ and $\psi(s)=\psi^{(0)}(s)$. 
This is obtained by calculating the right-hand side of \eqref{Eq_1101} 
using \eqref{Eq_101} and \eqref{Eq_103}.  
We have $\Psi_\omega(0)=0$ as well as $\Psi(t)$. 

As a special case, 
\[
\aligned 
\Psi_{1/2}(t) 
&= \,\frac{1}{2}(t+1)^2 - \frac{3}{2} + e^{-t}
\\ 
& \quad - 
\frac{t}{2}( \gamma_0 + 2\log 2 + \log \pi ) 
+
\frac{1}{4}\left[ \frac{\pi^2}{2} 
- e^{-t}\Phi\left(e^{-2t},2,\frac{1}{2}\right) \right] \\
&\quad  - \sum_{n \leq e^t} \frac{\Lambda(n)}{n}(t-\log n) 
\endaligned 
\]
for $t>0$. In this case, $g_{1/2}(t):=-\Psi_{1/2}(t)$ belongs to the class $\mathcal{G}_\infty$, 
because $Q_{1/2}(z)=i(\xi'/\xi)(1-iz)$ belongs to $\mathcal{N}$ 
and satisfies \eqref{Eq_106}. 
\bigskip

%

%

\begin{thebibliography}{99}
%
\bibitem{Bo01}
E. Bombieri,
\newblock{Remarks on Weil's quadratic functional in the theory of prime numbers. I}, 
\newblock{\it Atti Accad. Naz. Lincei Cl. Sci. Fis. Mat. Natur. Rend. Lincei (9) Mat. Appl.} 
\newblock{{\bf 11} (2000), no. 3, 183--233 (2001)}. 
%
\bibitem{BoLa99}
E. Bombieri, J. C. Lagarias, 
\newblock{Complements to Li's criterion for the Riemann hypothesis}, 
\newblock{\it J. Number Theory} 
\newblock{{\bf 77} (1999), no. 2, 274--287}. 
%
\comment{
\bibitem{Br05}
F. Brown, 
\newblock{Li's criterion and zero-free regions of $L$-functions}, 
\newblock{\it J. Number Theory} 
\newblock{{\bf 111} (2005), no. 1, 1--32}. {\color{red} 1箇所. 残す?}
}
%
\bibitem{Co99}
A. Connes, 
\newblock{Trace formula in noncommutative geometry and the zeros of the Riemann zeta function}, 
\newblock{\it Selecta Math. (N.S.)} 
\newblock{{\bf 5} (1999), no. 1, 29--106}.  
%
\bibitem{CC21a}
A. Connes, C. Consani, 
\newblock{Weil positivity and trace formula the archimedean place}, 
\newblock{\it Selecta Math. (N.S.)} 
\newblock{{\bf 27} (2021), no. 4, Paper No. 77, 70 pp}. 
%
\bibitem{CC21b}
A. Connes, C. Consani, 
\newblock{Spectral Triples and $\zeta$-Cycles}, 
\newblock{\url{https://arxiv.org/pdf/2106.01715.pdf}}. 
%
\bibitem{Con89}
J. B. Conrey, 
\newblock{More than two fifths of the zeros of the Riemann zeta function are on the critical line}, 
\newblock{\it J. Reine Angew. Math.} 
\newblock{{\bf 399} (1989), 1-26}. 
%
\bibitem{EMOT81}
A. Erd\'{e}lyi, W. Magnus, F. Oberhettinger, F. G. Tricomi, 
\newblock{Higher transcendental functions. {V}ol. {I}}, 
\newblock{Reprint of the 1953 original}, 
\newblock{\it Robert E. Krieger Publishing Co., Inc., Melbourne, Fla.}, 
\newblock{1981}. 
%
\bibitem{FRS12}
I. Fesenko, G. Ricotta, M. Suzuki,  
\newblock{Mean-periodicity and zeta functions}, 
\newblock{\it  Ann. Inst. Fourier (Grenoble)} 
\newblock{{\bf 62} (2012), no. 5, 1819--1887}.  
%
\bibitem{GGK01}
I. Gohberg, S. Goldberg, N. Krupnik, 
\newblock{Traces and determinants of linear operators}, 
\newblock{Operator Theory: Advances and Applications {\bf 116}}, 
\newblock{\it Birkh\"{a}user Verlag, Basel}, 
\newblock{2001}. 
%
\bibitem{GoKr69}
I. Gohberg, M. G. Kre\u{\i}n, 
\newblock{Introduction to the theory of linear nonselfadjoint operators}, 
\newblock{Translations of Mathematical Monographs, Vol. {\bf 18}}, 
\newblock{\it American Mathematical Society, Providence, R.I.}, 
\newblock{1969}. 
%
\comment{
\bibitem{HaWr60}
G. H. Hardy, E. M. Wright, 
\newblock{An introduction to the theory of numbers}, 
\newblock{Fourth edition}, 
\newblock{\it The Clarendon Press, Oxford University Press, New York}, 
\newblock{1960}. 
}
%
\bibitem{Ke92}
J. B. Keiper, 
\newblock{Power series expansions of Riemann's $\xi$ function}, 
\newblock{\it Math. Comp.} 
\newblock{{\bf 58} (1992), no. 198, 765--773}.  
%
\bibitem{Ko21}
S. Kotani, 
\newblock{Riemann's Zeta function and Krein's string}, unpublished but available at 
\url{https://www.researchgate.net/publication/348522988_Riemann's_Zeta_function_and_Krein's_string}. 
%
\bibitem{Kr52}
M. G. Kre\u{\i}n, 
\newblock{On a generalization of investigations of Stieltjes (Russian)}, 
\newblock{\it Doklady Akad. Nauk SSSR (N.S.)} 
\newblock{{\bf 87} (1952), 881--884}. 
%
%
\bibitem{KrLa77}
M. G. Kre\u{\i}n, H. Langer, 
\newblock{\"{U}ber einige {F}ortsetzungsprobleme, die eng mit der {T}heorie
              hermitescher {O}peratoren im {R}aume {$\Pi \sb{\kappa }$}
              zusammenh\"{a}ngen. {I}. {E}inige {F}unktionenklassen und ihre
              {D}arstellungen}, 
\newblock{\it Math. Nachr.} 
\newblock{{\bf 77} (1977), 187--236}. 
%
\bibitem{KrLa85}
M. G. Kre\u{\i}n, H. Langer, 
\newblock{On some continuation problems which are closely related to the
              theory of operators in spaces {$\Pi_\kappa$}. {IV}.
              {C}ontinuous analogues of orthogonal polynomials on the unit
              circle with respect to an indefinite weight and related
              continuation problems for some classes of functions}, 
\newblock{\it J. Operator Theory}  
\newblock{{\bf 13} (1985), no. 2, 299--417}. 
%
\bibitem{KrLa14}
M. G. Kre\u{\i}n, H. Langer, 
\newblock{Continuation of hermitian positive definite functions and related questions}, 
\newblock{\it Integral Equations Operator Theory}  
\newblock{{\bf 78} (2014), no. 1, 1--69}. 
%
\bibitem{KrNu77}
M. G. Kre\u{\i}n, A. A. Nudel'man, 
\newblock{The {M}arkov moment problem and extremal problems, 
              Ideas and problems of P. L. \v{C}eby\v{s}ev and A. A. Markov and their
              further development,
              Translated from the Russian by D. Louvish, 
              Translations of Mathematical Monographs, Vol. 50}, 
\newblock{\it American Mathematical Society, Providence, R.I.},   
\newblock{1977}. 
%
\bibitem{La99}
J. C. Lagarias,
\newblock{On a positivity property of the Riemann $\xi$-function}, 
\newblock{\it Acta Arith.}  
\newblock{{\bf 89} (1999), no. 3, 217--234}. 
%
\comment{
\bibitem{La06}
J. C. Lagarias,
\newblock{Hilbert spaces of entire functions and Dirichlet $L$-functions}, 
\newblock{\it Frontiers in number theory, physics, and geometry. I},   
\newblock{365--377}, 
\newblock{\it Springer, Berlin}, 
\newblock{2006}.　{\color{red} spectral interpretation の1箇所.}
}
%
\bibitem{Li97}
X.-J. Li,
\newblock{The positivity of a sequence of numbers and the Riemann hypothesis}, 
\newblock{\it J. Number Theory} 
\newblock{{\bf 65} (1997), no. 2, 325--333}. 
%
\bibitem{Lin17}
G. D. Lin,
\newblock{Recent developments on the moment problem}, 
\newblock{\it Journal of Statistical Distributions and Applications} 
\newblock{{\bf 4} (2017), Article number: 5}. 
%
\bibitem{Me05}
R. Meyer,
\newblock{A spectral interpretation for the zeros of the Riemann zeta function}, 
\newblock{\it Mathematisches {I}nstitut, {G}eorg-{A}ugust-{U}niversit\"{a}t
              {G}\"{o}ttingen: {S}eminars {W}inter {T}erm 2004/2005}, 
\newblock{117--137}, 
\newblock{\it Universit\"{a}tsdrucke G\"{o}ttingen, G\"{o}ttingen}, 
\newblock{2005}. 
%
\bibitem{NaSu22}
T. Nakamura, M. Suzuki, 
\newblock{On infinitely divisible distributions related to the Riemann hypothesis}, 
\newblock{submitted}. 
%
\bibitem{St76}
J. Stewart,
\newblock{Positive definite functions and generalizations, an historical survey}, 
\newblock{\it Rocky Mountain J. Math.} 
\newblock{{\bf 6} (1976), no. 3, 409--434}. 
%
\bibitem{Su12}
M. Suzuki,
\newblock{Two-dimensional adelic analysis and cuspidal automorphic representations of ${\rm GL}(2)$}, 
\newblock{\it Multiple Dirichlet series, L-functions and automorphic forms}, 
\newblock{339--361, Progr. Math., 300}, 
\newblock{\it Birkh\"{a}user/Springer, New York}, 
\newblock{2012}. 
%
\comment{
\bibitem{Su21}
M. Suzuki,
\newblock{Hamiltonians arising from $L$-functions in the Selberg class}, 
\newblock{\it J. Funct. Anal.}, 
\newblock{{\bf 281} (2021), no. 8, Paper No. 109116, 31 pp}. 
{\color{red} spectral interpretation の1箇所.}
}
%
\bibitem{Su22a}
M. Suzuki,
\newblock{The screw line of the Riemann zeta-function}, \\
\url{https://arxiv.org/abs/2209.04658}. 
%
\bibitem{Su22b}
M. Suzuki,
\newblock{Screw functions of Dirichlet series in the extended Selberg class}, \\
\url{https://arxiv.org/abs/2209.12832}. 
%
\bibitem{Su23a}
M. Suzuki,
\newblock{On the Hilbert space derived from the Weil distribution}, \\
\url{https://arxiv.org/abs/2301.00421}.
%
\bibitem{Su23b}
M. Suzuki,
\newblock{Li coefficients as norms of functions in a model space}, \\
\url{https://arxiv.org/abs/2301.05779}.
%
\bibitem{Tit86}
E. C. Titchmarsh,
\newblock{The theory of the Riemann zeta-function, Second edition, 
Edited and with a preface by D. R. Heath-Brown
}, 
\newblock{\it The Clarendon Press, Oxford University Press, New York}, 
\newblock{1986}. 
%
\bibitem{We52}
A. Weil,
\newblock{Sur les ``formules explicites'' de la th\'{e}orie des nombres
              premiers}, 
\newblock{\it Comm. S\'{e}m. Math. Univ. Lund [Medd. Lunds Univ. Mat. Sem.]},  
\newblock{{\bf 1952} (1952), Tome Suppl\'{e}mentaire, 252--265}. 
%
\bibitem{We72}
A. Weil,
\newblock{Sur les formules explicites de la th\'{e}orie des nombres}, 
\newblock{\it Izv. Akad. Nauk SSSR Ser. Mat.},  
\newblock{{\bf 36} (1972), 3--18.}. 
%
\bibitem{Wi41}
V. D. Widder,
\newblock{The Laplace Transform}, 
\newblock{Princeton Mathematical Series, vol. 6}, 
\newblock{\it Princeton University Press, Princeton, N. J.},  
\newblock{1941}. 
%
\comment{
\bibitem{Win14}
H. Winkler, 
\newblock{Two-dimensional Hamiltonian systems}, 
\newblock{\it Operator Theory}, 
\newblock{D. Alpay (eds.)},
\newblock{\it Springer, Basel}, 
\newblock{2015}, 
\newblock{pp. 525--547}. {\color{red} spectral interpretation の1箇所.}
}
%
\bibitem{Yo92}
H. Yoshida,
\newblock{On Hermitian forms attached to zeta functions}, 
\newblock{\it Zeta functions in geometry (Tokyo, 1990)}, 
\newblock{281--325, Adv. Stud. Pure Math., 21}, 
\newblock{\it Kinokuniya, Tokyo}, 
\newblock{1992}.
%
\end{thebibliography}
\end{document}